\documentclass[12pt]{article}
\usepackage[latin1]{inputenc}

\usepackage{amsmath}
\usepackage{amsfonts}
\usepackage{amssymb}
\usepackage{graphicx}
\usepackage{amsthm}
\usepackage{setspace}
\usepackage{booktabs}
\usepackage{tabularx}
\usepackage{lineno} 
\usepackage{color} 
\usepackage{enumerate}
\usepackage{bbm}
\usepackage{xcolor}
\usepackage[title]{appendix}
\usepackage{accents}
\usepackage{multirow}

\usepackage{thmtools}
\usepackage{thm-restate}
\usepackage{hyperref}
\usepackage{cleveref}

\usepackage[textsize=footnotesize,color=green!40]{todonotes}

\newcounter{i}
\newtheorem{theorem}{Theorem}[section]

\newtheorem{lemma}[theorem]{Lemma}

\newtheorem{proposition}[theorem]{Proposition}
\newtheorem{conjecture}[theorem]{Conjecture}

\newtheorem{corollary}[theorem]{Corollary}
\newtheorem{definition}[theorem]{Definition}

\newtheorem{claim}{Claim}[theorem]

    {\noindent \emph{Proof.} {}{#1}{}}{\hfill
    $\Diamond$\vspace{1em}}
    
\theoremstyle{plain} 
\newcommand{\thistheoremname}{}
\newtheorem{genericthm}[section]{\thistheoremname}

\newcommand{\cK}{\mathcal{K}}
\newcommand{\cOK}{\mathcal{OK}}

\newcommand{\F}{\mathcal{F}}
\newcommand{\C}{\mathcal{C}}
\newcommand{\cH}{\mathcal{H}}
\newcommand{\bR}{\mathbb{R}}
\newcommand{\N}{\mathbb{N}}

\title{Progress towards Nash-Williams' Conjecture on Triangle Decompositions}
\author{
Michelle Delcourt
\thanks{Department of Mathematics, Ryerson University,
Toronto, Ontario M5B 2K3, Canada {\tt mdelcourt@ryerson.ca}. Research supported by supported by NSERC under Discovery Grant No. 2019-04269.}
\and
Luke Postle
\thanks{Combinatorics and Optimization Department,
University of Waterloo, Waterloo, Ontario N2L 3G1, Canada {\tt lpostle@uwaterloo.ca}. Partially supported by NSERC
under Discovery Grant No. 2019-04304.}}
\date{\today}

\begin{document}
\font\smallrm=cmr8
\maketitle
\begin{abstract}
Partitioning the edges of a graph into edge disjoint triangles forms a \emph{triangle decomposition} of the graph.  A famous conjecture by Nash-Williams from 1970 asserts that any sufficiently large, triangle divisible graph on $n$ vertices with minimum degree at least $0.75 n$ admits a triangle decomposition.  In the light of recent results, the fractional version of this problem is of central importance.  A \emph{fractional triangle decomposition} is an assignment of non-negative weights to each triangle in a graph such that the sum of the weights along each edge is precisely 1.  

We show that for any graph on $n$ vertices with minimum degree at least $\left(\frac{7+\sqrt{21}}{14}\right)n \lessapprox 0.82733n$ admits a fractional triangle decomposition. Combined with results of Barber, K\"{u}hn, Lo, and Osthus, this implies that for all $\varepsilon > 0$, every sufficiently large triangle divisible graph on $n$ vertices with minimum degree at least $\left(\frac{7+\sqrt{21}}{14} + \varepsilon\right)n$ admits a triangle decomposition.
\end{abstract}

\section{Introduction}
A natural question in graph theory is whether the edges of a graph $G$ can be partitioned into edge disjoint copies of a small fixed subgraph $F$; such a partition of $E(G)$ is called an \emph{$F$-decomposition}.  Several necessary divisibility conditions arise for finding an $F$-decomposition: $e(F)$, the \emph{number of edges of $F$}, must divide $e(G)$, and $\gcd(F)$, the \emph{greatest common divisor of the degrees of the vertices of $F$}, must divide $\gcd(G)$.  If $F$ and $G$ satisfy these two divisibility conditions, then we say that $G$ is \emph{$F$-divisible}.  Although every graph with an $F$-decomposition must be $F$-divisible, not every $F$-divisible graph admits an $F$-decomposition.

In 1847 Kirkman~\cite{K} showed that when $G$ is a $K_3$-divisible complete graph, then $G$ admits a $K_3$-decomposition.  Over a century later, in the 1970s Wilson~\cite{W} generalized this by showing that for every graph $F$, every sufficiently large $F$-divisible complete graph admits an $F$-decomposition.  This result was a special case for graphs of the notorious Existence Conjecture of block designs dating from the mid-1800's.  In a recent major breakthrough result, Keevash~\cite{Keev} proved the Existence Conjecture using a mixture of algebraic and combinatorial techniques.  In~\cite{arb}, Glock, K\"{u}hn, Lo, and Osthus give a purely combinatorial proof of the Existence Conjecture via iterative absorption.

A natural, related area of study is finding $F$-decompositions in $F$-divisible (hyper)graphs with large minimum degree.  In fact, the results of Glock, K\"{u}hn, Lo, and Osthus~\cite{arb} extend to this much more general setting.  Perhaps the most famous conjecture in this minimum degree setting is the Nash-Williams Conjecture from 1970 which focuses on triangle decompositions as follows:

\begin{conjecture}[Nash-Williams~\cite{NW}]\label{ConjNW}
Let $G$ be a $K_3$-divisible graph with $n$ vertices and
minimum degree $\delta (G) \geq \frac{3}{4}n$. If $n$ is sufficiently large, then $G$ admits a $K_3$-decomposition.
\end{conjecture}
Many constructions show that $3/4$ is tight; for example, consider the following family of constructions from~\cite{min}. Let $H_1$ and $H_2$ be $(6k+2)$-regular graphs on $12k+6$ vertices. Consider the complete join of $H_1$ and $H_2$; this is a $K_3$-divisible graph with $n=24k+12$ vertices and minimum degree $12k+6 + 6k+2 = 18k +9 -1 = \frac{3}{4}n - 1$.  Any triangle must contain zero or two of the cross edges.  There are exactly $(12k+6)^2 = \frac{n^2}{4}$ cross edges belonging to $\frac{n^2}{8}$ triangles; however, $H_1$ and $H_2$ contain a total of $2 (3k+1)(12k+6) < \frac{n^2}{8}$ additional edges.  Thus, this graph does not admit a $K_3$-decomposition. 

For general cliques, a folklore generalization of the Nash-Williams Conjecture asserts that every sufficiently large, $K_r$-divisible graph $G$ on $n$ vertices with $\delta (G) \geq \frac{r}{r+1}n$ admits a $K_r$-decomposition.  If true, then this would also be tight (see Yuster~\cite{Y05} for a construction). 

\subsection{The Importance of Fractional Decompositions}
Recent breakthrough results of Barber, K\"{u}hn, Lo, and Osthus~\cite{tri16} and later Glock, K\"{u}hn, Lo, Montgomery, Osthus~\cite{decomp} show that the existence of $F$-decompositions is related to the existence of fractional $F$-decompositions as follows. 

A \emph{fractional $F$-decomposition} of $G$ is an assignment of non-negative weights to each copy of $F$ in $G$ such that the sum of the weights along each edge is precisely 1. The \emph{fractional $F$-decomposition threshold} $\delta^*_{F}$ is defined as $\displaystyle\limsup_{n \rightarrow \infty}\delta^*_{F}(n)$ where $\delta^*_{F}(n)$ is the least $c > 0$ such that any graph $G$ on $n$ vertices with minimum degree $\delta(G) > cn$ has a fractional $F$-decomposition. 

Combining the breakthrough results of Barber, K\"{u}hn, Lo, and Osthus~\cite{tri16} for $r=3$ and Glock, K\"{u}hn, Lo, Montgomery, and Osthus~\cite{decomp} for $r>3$, the following is known:

\begin{theorem}[\cite{tri16},\cite{decomp}]\label{CliqueDecomp}
Let $r\geq 3$ and $\varepsilon > 0$.  Any sufficiently large, $K_r$-divisible graph $G$ on $n$ vertices with minimum degree
$$\delta(G) \geq \left(\max\left\{\delta^*_{K_{r}}, \frac{r}{r+1}\right\}+\varepsilon\right)n$$
admits a $K_r$-decomposition.
\end{theorem}

There is an equivalent formulation of Theorem~\ref{CliqueDecomp} as follows. We define the \emph{decomposition threshold of $F$}, denoted $\delta_F$, as $\displaystyle\limsup_{n \rightarrow \infty}\delta_{F}(n)$ where $\delta_{F}(n)$ is the least $c > 0$ such that any $F$-divisible graph $G$ on $n$ vertices with minimum degree $\delta(G) > cn$ has an $F$-decomposition. Then Theorem~\ref{CliqueDecomp} is equivalent to 
$$\delta_{K_r} = \max\left\{\delta^*_{K_r}, \frac{r}{r+1}\right\}.$$

The proof of Theorem~\ref{CliqueDecomp} uses the method of iterative absorption to transform an approximate clique decomposition of a graph into a clique decomposition, while an earlier result of Haxell and R\"{o}dl~\cite{HR} shows how to transform a fractional clique decomposition into an approximate clique decomposition. Thus determining $\delta^*_{K_r}$ is now the key to determining $\delta_{K_r}$.  We note that Yuster's constructions~\cite{Y05} mentioned above imply that $\delta^*_{K_r} \geq r/(r+1) = 1 -1/(r+1)$, and hence $\delta^*_{K_3} \geq 3/4$.  Showing that $\delta^*_{K_3} \leq 3/4$ would prove the Nash-Williams Conjecture asymptotically.

As for general graphs $F$, the case when $F$ is bipartite has been completely determined by Glock, K\"{u}hn, Lo, Montgomery, and Osthus~\cite{decomp} (in particular, $\delta_F$ is either $2/3$ or $1/2$ depending on the structure of $F$). As for other $F$, it turns out that the chromatic number $\chi(F)$ is of fundamental importance as the following general result of Glock, K\"{u}hn, Lo, Montgomery, and Osthus in~\cite{decomp} shows:

\begin{theorem}[Glock, K\"{u}hn, Lo, Montgomery, and Osthus~\cite{decomp}]\label{ChiDecomp}
Let $\varepsilon > 0$.  Let $F$ be a graph with chromatic number $\chi = \chi(F) \geq 3$, then any sufficiently large, $F$-divisible graph $G$ on $n$ vertices with minimum degree
$$\delta(G) \geq \left(\max\left\{\delta^*_{K_{\chi}}, \frac{\chi}{\chi+1}\right\}+\varepsilon\right)n$$
admits an $F$-decomposition.
\end{theorem}

Again, this is equivalent to $$\delta_F \le \max \left\{ \delta^*_{K_{\chi(F)}}, \frac{\chi(F)}{\chi(F)+1}\right\}.$$ 
\noindent Thus $\delta^*_{K_r}$ determines not only the decomposition threshold of cliques but provides an upper bound on the decomposition threshold of all $r$-chromatic graphs. Given these results, determining $\delta^*_{K_r}$ is now of central importance in this area. 

For general $r$, Yuster~\cite{Y05} in 2005 showed that $\delta^*_{K_r} \leq 1 - 1/(9r^{10})$. In 2012, Dukes~\cite{Dukes} improved this to $\delta^*_{K_r} \leq 1 - 2/(9r^2(r-1)^2)$, which was then further improved by Barber, K\"{u}hn, Lo, Montgomery, and Osthus~\cite{decomp} to $\delta^*_{K_r} \leq 1-1/(10^4r^{1.5})$. The current best known bound is a more recent improvement due to Montgomery~\cite{Montgomery} that $\delta^*_{K_r} \leq 1- 1/(100r)$, which is of the same order in $r$ as the known lower bound of $1 -1/(r+1)$.

Even better bounds are known for the triangle case.  In her thesis in 2014, Garaschuk~\cite{G14} showed that $\delta^*_{K_3} \leq 0.956$. In 2015, Dross~\cite{D15} proved the current best known bound that $\delta^*_{K_3} \leq 0.9$ using the min-flow max-cut theorem.

\subsection{Our Main Results}

Our main theorem is the following significant improvement on Dross' result:

\begin{theorem}\label{thm:main} Let $G$ be a graph on $n$ vertices with
minimum degree $\delta(G)\geq \left(\frac{7+\sqrt{21}}{14}\right)n$, then $G$ admits a fractional $K_3$-decomposition.
\end{theorem}

Note that $\left(\frac{7+\sqrt{21}}{14}\right)n< 0.82733n$.  Combined with Theorem~\ref{CliqueDecomp}, our result gives the following progress on the Nash-Williams Conjecture:

\begin{corollary}\label{cor:main}
Let $\varepsilon > 0$. Let $G$ be a $K_3$-divisible graph with $n$ vertices and
minimum degree $\delta(G)\geq \left(\frac{7+\sqrt{21}}{14} + \varepsilon\right)n$.  If $n$ is sufficiently large, then $G$ admits a $K_3$-decomposition.
\end{corollary}

Combined with Theorem~\ref{ChiDecomp}, our result gives the following more general corollary:

\begin{corollary}\label{CorChi}
Let $\varepsilon > 0$. Let $F$ be a graph with chromatic number $\chi(F)=3$, then any sufficiently large, $F$-divisible graph $G$ on $n$ vertices with minimum degree $\delta(G) \ge \left(\frac{7+\sqrt{21}}{14} + \varepsilon\right)n$ admits an $F$-decomposition.
\end{corollary}

\noindent Independently around the same time Dukes and Horsley~\cite{DH20} announced a value of $0.852$ for $\delta^*_{K_3}$.  We should mention that their proof depends on the use of computer programs whereas ours is completely verifiable by hand. Interestingly they also provide examples to demonstrate that their approach (as well as Dross's approach~\cite{D15}) encounters a theoretical barrier at $5/6 \approx 0.83333 > 0.82733$.

Our main result combined with the work of Condon, Kim, K\"uhn, and Osthus (see Corollary 1.4 in~\cite{band}) immediately gives the following corollary. Here we say that
a collection $\cH = \left\{H_1, \ldots, H_s\right\}$ of graphs \emph{packs} into $G$ if there exist pairwise edge-disjoint copies
of $H_1, \ldots, H_s$ in $G$, and $\Delta(G)$ is defined to be the \emph{maximum degree} of a graph $G$.

\begin{corollary}
For all $\Delta, k \in \N\setminus\left\{1\right\}$ and $0 < \nu, \delta < 1$,
there exist $\xi > 0$ and $n_0 \in \N$ such that for $n \geq n_0$ the following holds for every $n$-vertex graph $G$
with
$$(\delta - \xi)n \leq \delta(G) \leq \Delta(G) \leq (\delta + \xi)n.$$
\begin{enumerate}
\item Let $F$ be an $n$-vertex graph consisting of a union of vertex-disjoint cycles and let $\F$ be a collection of copies 
of F.  Further suppose $\delta> \frac{7+\sqrt{21}}{14}$ and $e(\F) \leq (1-\nu)e(G)$.  Then $\F$ packs into $G$.
\item Let $\C$ be a collection of cycles, each on at most $n$ vertices.  Further suppose $\delta > \frac{7+\sqrt{21}}{14}$ and
$e(\C) \leq (1-\nu)e(G)$.  Then $\C$ packs into $G$.
\end{enumerate}
\end{corollary}

The next corollary for regular graphs follows immediately from our main result combined with the work of Condon, Kim, K\"uhn, and Osthus~\cite{band} (see Corollary 2.8 in~\cite{BCC} by Glock, K\"{u}hn, and Osthus):

\begin{corollary}
For all $\varepsilon > 0$, the following holds for sufficiently large $n$. Assume that $\F$ is a collection of 2-regular $n$-vertex graphs. Assume that $G$ is a $d$-regular $n$-vertex graph with $d \geq \left(\frac{7+\sqrt{21}}{14} + \varepsilon\right)n$. If $e(F)\leq(1 - \varepsilon)e(G)$, then $\F$ packs into $G$.
\end{corollary}

Our approach for proving the fractional version is novel and differs from previous work on the Nash-Williams Conjecture in that we introduce two new concepts specifically developed for this problem that we refer to as \emph{delegation} and \emph{cancellation}.  We also rely on \emph{edge-gadgets} as introduced by Barber, K\"{u}hn, Lo, Montgomery, and Osthus~\cite{tri17} (a weight function for $K_r$s contained in a $K_{r+2}$) as well as tools from nonlinear optimization. We describe our two new ideas and overview the proof in the next section before proceeding to outline the rest of the paper.

\section{Overview of the Proof}

First in Subsection~\ref{OverviewIdeas}, we provide an overview of the ideas involved in the proof of Theorem~\ref{thm:main}. Then we outline the remainder of the paper in Subsection~\ref{Outline}.

\subsection{Overview}\label{OverviewIdeas}

Recall that a \emph{fractional triangle decomposition} is an assignment of non-negative weights to each triangle in a graph such that the sum of the weights along each edge is precisely 1.  An \emph{edge-gadget} (see Definition~\ref{def:gadget}) is a local redistribution of the weights of the triangles in a $K_5$ containing a given edge $e$ so as to increase only the weight of $e$ (while leaving the weights of all other edges unchanged).  Given any current triangle weighting, one can use edge-gadgets to satisfy any remaining demands of edges; each edge-gadget yields both some positive and some negative modifications for the triangle weights. Hence, the overuse of edge-gadgets could result in the final weight of some triangle being negative and hence the weighting not corresponding to a fractional triangle decomposition.   

When restricted to the case of triangle decompositions, Barber, K\"{u}hn, Lo, Montgomery, and Osthus' edge-gadget proof in~\cite{tri17} begins with a uniform initial positive weighting on triangles and distributes the remaining demand of each edge $e$ \emph{uniformly} over the edge-gadgets containing $e$.  We, however, use a \emph{non-uniform} distribution of the demand of each edge $e$ over the edge-gadgets containing $e$.  

Curiously, our method works identically for any uniform initial weighting of the triangles.  Hence for ease of reading we initialize the weights to be 0 (equivalently, we do not use any initial weighting).
 Thus the remaining demand on any edge is equal to its initial demand, namely 1.  We then \emph{delegate} this demand first through the triangles containing that edge, then through the $K_4$s containing each of those triangles, and then through the $K_5$s containing those $K_4$s. This defines a weighting on the edge-gadgets and in turn a weighting on the triangles. Note that such a delegation is only well-defined since our choice of minimum degree is strictly greater than $\frac{3}{4}n$ (and hence every edge is in a $K_5$). 

By virtue of the delegation process, it is clear that the final triangle weighting yields a weight of 1 across each edge.  Moreover, as detailed in the majority of the paper, the weight on each triangle is non-negative given our choice of minimum degree, and hence we obtain the desired fractional triangle decomposition.  To show the final weights of the triangles are non-negative requires a fair amount of work and the use of non-linear optimization. 

The key concept we invoke to verify this we refer to as \emph{cancellation} which we describe as follows. Cancellation is an attempt to pair over each triangle $T$ an edge-gadget (or set of edge-gadgets) from which $T$ receives negative modification to an edge-gadget (or set of edge-gadgets) from which $T$ receives positive modification; intuitively these contributions should mostly cancel out, leaving it easier to show a non-negative final weight. Crucially we perform these pairings only at the triangle level before demands are delegated to the $K_4$s and $K_5$s.  From then on in the process, these opposing demands effectively \emph{cancel out} since that triangle will delegate the demands of its edges uniformly to the $K_4$s containing it and subsequently those $K_4$s will delegate uniformly to the $K_5$s containing them.

Indeed, this idea is not just intuition, we formally make use of this as follows. The proof that each triangle has non-negative final weight proceeds by setting up a related maximization program to be solved. After symmetrizing the variables, we are left with a 10 variable non-linear optimization program  whose objective value is the sum of three terms. Each term is some ratio of positive factors times a difference of two variables (each difference corresponding to a cancellation of two edge-gadgets). We then form a new program by replacing each term with its \emph{ramp} (i.e.~the maximum of itself and 0). This ensures that each term is now non-negative which is essential to our solution of the program.

To solve the 10 variable program, we proceed to reduce the number of variables, first to six, then one at a time in the right order, crucially using the fact that all the terms now have non-negative factors (in fact strictly positive for those factors in the denominator). This makes the reductions fairly straightforward if tedious. The final two variables are the hardest to reduce. However at that stage, the terms are in fact guaranteed to be non-negative and so the ramps are no longer necessary. We then use partial derivatives to find a maximum point. Finally we evaluate the objective function at this maximum point; indeed, our value of minimum degree is precisely the point where the final triangle weights are non-negative.

\subsection{Outline of Paper}\label{Outline}

In Section~\ref{GadgetsAndMain}, we define edge-gadgets and our weighting of them. We state our main technical theorem (Theorem~\ref{RealMain}) and then prove Theorem~\ref{thm:main} assuming Theorem~\ref{RealMain}. Finally, we reformulate the problem for convenience during optimization.
 
In Section~\ref{Opt}, we formally state this problem as a maximization program and use a symmetrization argument to reduce the number of variables. In Section~\ref{Solving}, we formulate the new ramping program mentioned in the subsection above and solve said program by slowly reducing the number of variables. We conclude that section by proving Theorem~\ref{RealMain}.

Finally in Section~\ref{Further}, we discuss how our methods could be used to make further improvements on the Nash-Williams Conjecture.

\section{Edge-Gadgets and Proof of Main Theorem}\label{GadgetsAndMain}

In Subsection~\ref{Gadgets}, we define edge-gadgets formally. In Subsection~\ref{Weighting}, we formally define our weighting of the edge-gadgets. We prove how the weighting yields a weight of $1$ on each edge. We then state our main technical theorem (Theorem~\ref{RealMain}) that the resulting weight on triangles is non-negative for our choice of minimum degree and then prove Theorem~\ref{thm:main} assuming Theorem~\ref{RealMain}. In Subsection~\ref{Reform}, we then reformulate this theorem in a more manageable form that involves cancellation (while also transforming it into a maximization problem).

\subsection{Edge-Gadgets}\label{Gadgets}

First we formalize some notation for the set of cliques containing a given smaller clique as follows.

\begin{definition}
Let $G$ be a graph. We let $\cK_{\ell}(G)$ denote the set of cliques in $G$ on exactly $\ell$ vertices. For a subgraph $H \subseteq G$, we let $\cK_{\ell}(G,H)$ denote the set of elements in $\cK_{\ell}(G)$ that contain $H$ as a subgraph.  For $S\subseteq V(G)$, we let $\cK_{\ell}(G,S) := \cK_{\ell}(G,G[S])$.
\end{definition}

Next we formally define fractional triangle decomposition in terms of weightings.

\begin{definition}

A \emph{fractional triangle decomposition} of a graph $G$ is, equivalently to the definition given before,
a {\bf non-negative} function  $w: \cK_3(G) \rightarrow \bR$ such that for every $e\in E(G)$, $$\sum_{T\in \cK_3(G, e)} w(T) = 1.$$
\end{definition}

Now we present the definition of an \emph{edge-gadget} introduced by Barber, K\"{u}hn, Lo, Montgomery, and Osthus~\cite{tri17}, as follows.

\begin{definition}\label{def:gadget}
Let $G$ be a graph. For $K\in \cK_5(G)$ and $e\in E(K)$, let $E_i(K,e) = \left\{f \in E(K): |e \cap f|=i\right\}$ and $T_j(K,e) = \left\{T \in \cK_3(K): |e \cap T|=j\right\}$. \\

The \emph{edge-gadget} of $e$ in $K$ is a function $\displaystyle\psi_{K,e}: \cK_3(G) \rightarrow \bR$ with 
 \begin{equation*}
  \psi_{K,e}(T) = \begin{cases}
        +\frac{1}{3}, \text{ if } T\in T_0(K,e),\\
				-\frac{1}{6}, \text{ if } T\in T_1(K,e),\\
				+\frac{1}{3}, \text{ if } T\in T_2(K,e),\text{ and}\\
				0, \text{ otherwise.}
        \end{cases}
 \end{equation*}
\end{definition}

Edge-gadgets are useful in that they assign a non-zero weight (scaled to be 1) to precisely one edge as the next proposition notes.

\begin{proposition}\label{EdgeGadgetWeight}
Let $e \in E(G)$ and $K \in \cK_5(G)$.  If $f\in E(G)$, then 
 \begin{equation*}
  \sum_{T \in \cK_3(G,f)}\psi_{K,e}(T) = \begin{cases}
        1, \text{ if } f=e,\text{ and}\\
				0, \text{ otherwise.}
        \end{cases}
 \end{equation*}
\end{proposition}
\begin{proof}

Let
$S(f) = \sum_{T \in \cK_3(G,f)}\psi_{K,e}(T).$  If $f=e$, then $f\in E_2(K,e)$ and hence $S(f) = 3\cdot\frac{1}{3}=1$ as desired. So suppose $f\ne e$. If $f\in E_1(K,e)$, then $S(f) = \frac{1}{3} -\frac{1}{6}\cdot 2 =0$ as desired. 
If $f\in E_0(K,e)$, then $S(f) = -\frac{1}{6} \cdot 2 + \frac{1}{3} = 0$ as desired. 
Finally if $f\in E(G)\setminus E(K)$, then $S(f) = 0$ as desired.
\end{proof}

\subsection{Our Weighting}\label{Weighting}

The proofs in~\cite{tri17} and~\cite{D15} begin with an essentially uniform initial weighting of copies of $K_3$ and via local moves use the edge-gadgets to obtain a fractional $K_3$-decomposition.  Using a random process, Montgomery~\cite{Montgomery} instead starts with an initial weighting that is closer to a fractional $K_3$-decomposition.  In this work, we pick our initial weighting in a different way and utilize cancellations to obtain a fractional $K_3$-decomposition.

We define a weight of an edge-gadget $K$ in the following way. Instead of distributing uniformly over copies of $K_5$ containing $e$, we do the following.  We distribute uniformly over $\cK_3(G,e)$.  Then for each $T \in \cK_3(G,e)$ we distribute uniformly over $\cK_4(G,T)$.  Finally for each $K \in \cK_4(G,T)$, we distribute uniformly over $\cK_5(G,K)$. (This is the \emph{delegation} described before).

To formalize this, we need the following definitions. First, we need ordered cliques as follows.

\begin{definition}
Let $G$ be a graph. An \emph{ordered $r$-clique} of $G$ is an $r$-tuple ($v_1,v_2,\ldots, v_r$) such that $v_1,\ldots, v_r \in V(G)$ and $G[\{v_1,\ldots, v_r\}] \in \cK_r(G)$. We let $\cOK_{r}(G)$ denote the set of ordered $r$-cliques in $G$. If $K=(v_1,\ldots, v_r)\in \cOK_r(G)$, then we let $V(K) = \{v_1,\ldots, v_r\}$.
\end{definition}

Next we need some notation for the ordered cliques containing a subgraph (or set of vertices).
\begin{definition}
Let $G$ be a graph. For a subgraph $H \subseteq G$, we let $\cOK_{r}(G,H)$ denote the set of elements $K\in \cOK_{r}(G)$ such that $V(H)\subseteq V(K)$.  For $S\subseteq V(G)$, we let $\cOK_{r}(G,S) := \cK_{r}(G,G[S])$.
\end{definition}

Then we need to define containing an ordered subgraph and the set of ordered cliques containing a smaller ordered clique as an ordered subgraph.

\begin{definition}
Let $G$ be a graph and $s\ge r \ge 1$. Let $H_1= (v_1,\ldots, v_s)\in \cOK_s(G)$ and $H_2=(u_1,\ldots,u_r) \in \cOK_r(G)$. We say $H_1$ is an \emph{ordered subgraph} of $H_2$ if $u_1\ldots u_r$ is a (not necessarily consecutive) subsequence of $v_1\ldots v_s$. For an ordered $r$-clique $H \subseteq G$, we let for every $s\ge r$, $\cOK_{s}(G,H)$ denote the set of elements in $\cOK_{s}(G)$ that contain $H$ as an ordered subgraph. 
\end{definition}

We are now ready to define a weight on ordered cliques as follows.

\begin{definition}
Let $G$ be a graph and let $r\in \{2,3,4\}$. For every $K=(v_1,\ldots, v_r)\in \cOK_r(G)$, we define a weight 

$$W(K) = \prod_{i=2}^r \frac{1}{|\cK_{i+1}(G, \{v_1, \ldots, v_i\})|}.$$

\noindent For ease of reading, we will let $W(v_1,\ldots, v_r) := W(K)$.
\end{definition}

We also need to extend $\psi_{K,e}(T)$ to ordered cliques $K$ and to ordered triangles $T$ (these will have the same value; this is just for convenience).

\begin{definition}
Let $G$ be a graph. If $K=(v_1,\ldots, v_5) \in \cOK_5(G)$ and $T\in \cK_3(G)$, then we define $\psi_{K}(T) := \psi_{G[V(K)],v_1v_2}(T)$. Similarly if $O\in \cOK_3(G)$, then we define $\psi_{K}(O) := \psi_{K}(G[V(O)])$.
\end{definition}

We are now ready to define our weight function on triangles as follows.

\begin{definition}\label{def:weight}
Let $G$ be a graph.  We define a function $w_G:\cK_3(G) \rightarrow \bR$ as
$$w_G(T) := \frac{1}{2} \cdot \sum_{K=(v_1,\ldots, v_5)\in \cOK_5(G)} W(v_1,\ldots, v_4) \cdot \psi_{K}(T).$$
\end{definition}

The following proposition shows that our weighting of the triangles yields a weight of 1 on each edge.

\begin{proposition}\label{PropWeight}
Let $G$ be a graph with minimum degree $\delta(G) > \frac{3}{4} \cdot v(G)$.  If $e \in E(G)$, then 
$$\sum_{T \in \cK_3(G,e)}w_G(T)=1.$$
\end{proposition}
\begin{proof}
Note that as $\delta(G) > \frac{3}{4} \cdot v(G)$, we have for every $r\in\{2,3,4\}$ and $S\in \cOK_r(G)$ that $W(S)$ is well-defined and strictly positive.

Now let $W_e = \sum_{T \in \cK_3(G,e)}w_G(T)$. Using the definition of $w_G(T)$, we find that

$$W_e = \sum_{T \in \cK_3(G,e)}\Bigg( \frac{1}{2} \cdot \sum_{K=(v_1,\ldots, v_5)\in \cOK_5(G)} W(v_1,\ldots, v_4) \cdot \psi_{K}(T) \Bigg).$$

\noindent Rearranging sums, we find that

$$W_e = \frac{1}{2} \cdot \sum_{K=(v_1,\ldots, v_5)\in \cOK_5(G)} W(v_1,\ldots, v_4) \Bigg(\sum_{T\in \cK_3(G,e)} \psi_{K}(T) \Bigg).$$

\noindent By Proposition~\ref{EdgeGadgetWeight}, $\sum_{T \in \cK_3(G,e)} \psi_{K}(T) = 1$ if $e=v_1v_2$ and $0$ otherwise. Hence, we have that

\begin{align*}
W_e &= \frac{1}{2} \mathop{\sum_{K=(v_1,\ldots, v_5)\in \cOK_5(G):}}_{e=v_1v_2} W(v_1,\ldots, v_4)=\frac{1}{2} \mathop{\sum_{K=(v_1,\ldots, v_4)\in \cOK_4(G):}}_{e=v_1v_2} W(v_1,\ldots, v_3)\\
&=\frac{1}{2} \mathop{\sum_{K=(v_1, v_2, v_3)\in \cOK_3(G):}}_{e=v_1v_2} W(v_1, v_2)=\frac{1}{2} \mathop{\sum_{K=(v_1, v_2)\in \cOK_2(G):}}_{e=v_1v_2} 1= 1,
\end{align*}

\noindent as desired.
\end{proof}

Hence $w_G$ is a fractional triangle decomposition provided that $w_G$ is non-negative. Thus the remainder of the paper is devoted to proving the following result.

\begin{theorem}\label{RealMain}
Let $G$ be a graph. If $\delta(G) \ge (1-d)v(G)$ where $d=\frac{7 - \sqrt{21}}{14} > 0.17267$, then for every $T\in \cK_3(G)$,

$$w_G(T)\ge 0.$$
\end{theorem}

Assuming Theorem~\ref{RealMain}, we are now able to prove Theorem~\ref{thm:main}.

\begin{proof}[Proof of Theorem~\ref{thm:main}]
By Proposition~\ref{PropWeight}, for every $e \in E(G)$, we have that $$\sum_{T \in \cK_3(G,e)}w_G(T)=1.$$ By Theorem~\ref{RealMain}, $w_G$ is non-negative. Hence $w_G$ is a fractional triangle decomposition of $G$ as desired.
\end{proof}

\subsection{Reformulation}\label{Reform}

In fact, we prove a stronger theorem than Theorem~\ref{RealMain} as follows. First we define a weight function on ordered triangles.

\begin{definition}
We define a function $w_G:\cOK_3(G) \rightarrow \bR$ as
$$w_G(O) := \frac{1}{2} \cdot \sum_{K=(v_1,\ldots, v_5)\in \cOK_5(G,O)} W(v_1,\ldots, v_4) \cdot \psi_{K}(O).$$
\end{definition}

Clearly if $T\in \cK_3(G)$, then $w_G(T) = \sum_{O\in \cOK_3(G,T)} w_G(O).$ Hence to prove Theorem~\ref{RealMain}, it suffices to prove the following:

\begin{theorem}\label{RealMain2}
Let $G$ be a graph. If $\delta(G) \ge (1-d)v(G)$ where $d=\frac{7 - \sqrt{21}}{14} > 0.17267$, then for every $O\in \cOK_3(G)$,
$$w_G(O)\ge 0.$$
\end{theorem}

The key idea to proving Theorem~\ref{RealMain2} is to collect the terms in $w_G(O)$ according to how $O$ appears as a subsequence of $K$. In particular, we will then pair the terms which have the same set of vertices in their first three positions as follows. (This is the \emph{cancellation} described earlier.)

We may now rewrite $w_G(O)$ as follows.
\begin{lemma}\label{recast}
If $O=(x_1,x_2,x_3) \in \cOK_3(G)$ and $R=\bigcap_{i=1}^3 N(x_i)$, then
\begin{align*} w_G(O) = &\frac{1}{6} \Bigg( W(x_1,x_2) -\sum_{y\in R} \Bigg( W(x_1,y,x_2) - W(x_1,x_2,y) \\
&+ \sum_{z\in N(y)\cap R}\Bigg( W(x_1,y,x_2,z) - W(x_1,x_2,y,z) + W(x_1,y,z,x_2) - W(z,y,x_1,x_2)\Bigg) \Bigg) \Bigg).
\end{align*}
\end{lemma}
\begin{proof}
By definition

$$w_G(O) := \frac{1}{2} \cdot \sum_{K=(v_1,\ldots, v_5)\in \cOK_5(G,O)} W(v_1,\ldots, v_4) \cdot \psi_{K}(O).$$

\noindent Yet for every $K=(v_1,\ldots, v_5)\in \cOK_5(G,O)$, we have by Proposition~\ref{EdgeGadgetWeight} that $\psi_K(O) = +\frac{1}{3}$ if $|V(O)\cap \{v_1,v_2\}|\in \{0,2\}$ and $\psi_K(O) = -\frac{1}{6}$ if $|V(O)\cap \{v_1,v_2\}|=1$. Thus, we separating by the possible subsequences for $O$, we have the following

\begin{align*} w_G(O) = &\frac{1}{2} \sum_{y\in R}\sum_{z\in N(y)\cap R}\cdot \Bigg( W(x_1,x_2,x_3,y) \left(+\frac{1}{3}\right) \\
&+ W(x_1,x_2,y,x_3) \left(+\frac{1}{3}\right) + W(x_1,y,x_2,x_3) \left(-\frac{1}{6}\right) + W(y,x_1,x_2,x_3) \left(-\frac{1}{6}\right)\\
&+ W(x_1,x_2,y,z) \left(+\frac{1}{3}\right) + W(x_1,y,x_2,z) \left(-\frac{1}{6}\right) + W(y,x_1,x_2,z) \left(-\frac{1}{6}\right)\\
&+ W(y,z,x_1,,x_2) \left(+\frac{1}{3}\right) + W(x_1,y,z,x_2) \left(-\frac{1}{6}\right) + W(y,x_1,z,x_2) \left(-\frac{1}{6}\right) \Bigg).
\end{align*}

\noindent Yet, by symmetry we have that $W(x_1,y,x_2,x_3) = W(y,x_1,x_2,x_3)$. Similarly, $W(x_1,y,x_2,z) = W(y,x_1,x_2,z)$ and $W(x_1,y,z,x_2) = W(y,x_1,z,x_2)$. Thus

\begin{align*} w_G(O) = &\frac{1}{6} \sum_{y\in R}\sum_{z\in N(y)\cap R} \Bigg( W(x_1,x_2,x_3,y) + W(x_1,x_2,y,x_3) - W(x_1,y,x_2,x_3)\\
&+ W(x_1,x_2,y,z) - W(x_1,y,x_2,z) + W(y,z,x_1,,x_2) - W(x_1,y,z,x_2) \Bigg).
\end{align*}

\noindent Note that the first three terms do not depend on $z$. Hence when summing over $z$, we may instead multiply by a factor of $|N(y)\cap R| = |\cK_5(G,\{y,x_1,x_2,x_3\})|$. Yet by definition, $W(x_1,x_2,x_3,y) \cdot |\cK_5(G,\{y,x_1,x_2,x_3\})| = W(x_1,x_2,x_3)$. Similarly, $W(x_1,x_2,y,x_3) \cdot |\cK_5(G,\{y,x_1,x_2,x_3\})| = W(x_1,x_2,y)$ and $W(x_1,y,x_2,x_3) \cdot |\cK_5(G,\{y,x_1,x_2,x_3\})| = W(x_1,y,x_2)$. Thus

\begin{align*} w_G(O) = &\frac{1}{6} \sum_{y\in R}\Bigg( W(x_1,x_2,x_3) + W(x_1,x_2,y) - W(x_1,y,x_2)\\
+&\sum_{z\in N(y)\cap R} \Bigg( W(x_1,x_2,y,z) - W(x_1,y,x_2,z) + W(y,z,x_1,,x_2) - W(x_1,y,z,x_2) \Bigg) \Bigg).
\end{align*}

Finally, we note that $W(x_1,x_2,x_3)$ does not depend on $y$. When summing over $y$, we may instead multiply by a factor of $|R| = |\cK_4(G,\{x_1,x_2,x_3\})|$. Yet, by definition $W(x_1,x_2,x_3) \cdot |\cK_4(G,\{x_1,x_2,x_3\})| = W(x_1,x_2)$. Thus the formula now follows as desired.
\end{proof}

It is more convenient during optimization to use the following function related to $w_G(O)$.

\begin{definition}
We define a function $w_{G,1}:\cOK_3(G) \rightarrow \bR$ as follows: for each $O=(x_1,x_2,x_3)\in \cOK_3(G)$, let
\begin{align*}
w_{G,1}(O) &:= 1-|\cK_3(G,\{x_1,x_2\})| \cdot 6 \cdot w_G(O)\\
&= |\cK_3(G,\{x_1,x_2\})| \cdot \sum_{y\in R} \Bigg( W(x_1,y,x_2) - W(x_1,x_2,y) \\
& + \sum_{z\in N(y)\cap R}\Bigg( W(x_1,y,x_2,z) - W(x_1,x_2,y,z) + W(x_1,y,z,x_2) - W(z,y,x_1,x_2)\Bigg) \Bigg).
\end{align*}
\end{definition}

Note that if $\delta(G)\ge \frac{3}{4}v(G)$, then $|\cK_3(G,\{x_1,x_2\})| > 0$ for every $x_1x_2\in E(G)$. Thus, in order to prove Theorem~\ref{RealMain2}, it suffices now to prove the following theorem.

\begin{theorem}\label{MainWThm}
Let $G$ be a graph with $\delta(G) \ge (1-d)v(G)$ where $d=\frac{7-\sqrt{21}}{14}$. If $O=(x_1,x_2,x_3)\in \cOK_3(G)$, then

$$w_{G,1}(O) \le 1.$$ 
\end{theorem}

\section{Optimization}\label{Opt}

Clearly Theorem~\ref{MainWThm} is equivalent to some maximization program. Before stating the program, we first develop notation for the relevant variables and collect some necessary bounds in Subsection~\ref{Bounds}. In Subsection~\ref{Program}, we state our program and reformulate it in terms of variables instead of graphs. In Subsection~\ref{ReduceTo10}, we use a symmetrization argument to reduce from an arbitrary number of variables to just 10 variables.

In Section~\ref{Solving}, we solve the 10 variable program as follows.  In Subsection~\ref{ReduceTo6}, we then upper bound the program with a new program that uses ramps of functions (i.e.~the maximum of a function and 0). This is the key that allows us to slowly reduce the number of variables in the remainder of Section~\ref{Solving} until we solve the program.

\subsection{More Notation and Bounds}\label{Bounds}

\begin{definition}

Let $S\subseteq V(G)$. The \emph{common neighbor density} of $S$ is defined as

$$\hat{N}(S) := \frac{\left|V(G) \setminus \left(\bigcup_{s\in S} N(s)\right)\right|}{v(G)}.$$

Similarly if $H$ is a subgraph of $G$, we define the \emph{common neighbor density} of $H$, denoted $\hat{N}(H)$, as equal to $\hat{N}(V(H))$, the common neighbor density of $V(H)$.
\end{definition}

Note that for $S\subseteq V(G)$ such that $G[S]\in \cK_{|S|}(G)$, we have that

$$|\cK_{|S|+1}(G,S)| = v(G) \cdot \hat{N}(S).$$

Note that $\hat{N}(\emptyset)=1$ and that for each $v\in V(G)$, $\hat{N}(v)$ is the normalized degree of $v$ in $G$. We note the following bounds on $\hat{N}$, the first relates a set and its subset (i.e.~that $\hat{N}$ is monotone decreasing), the second relates sets with their intersection and union (i.e.~that $\hat{N}$ is supermodular).

\begin{proposition}\label{SubsetBound}
Let $G$ be a graph. If $S\subseteq S' \subseteq V(G)$, then

$$\hat{N}(S) \ge \hat{N}(S').$$
\end{proposition}
\begin{proof}
This follows since 

$$\bigcap_{s\in S} N(s) \supseteq \bigcap_{s\in S'} N(s).$$
\qedhere
\end{proof}

\begin{proposition}\label{IntersectionBound}
Let $G$ be a graph.  If $A,B\subseteq V(G)$, then

$$\hat{N}(A\cup B) \ge \hat{N}(A) + \hat{N}(B) - \hat{N}(A\cap B).$$
\end{proposition}
\begin{proof}
Let $X = V(G)\setminus \bigcup_{a\in A} N(a)$ and $Y = V(G)\setminus \bigcup_{b\in B} N(b)$. Now, 

$$|X\cap Y| = |X|+|Y|-|X\cup Y|.$$

\noindent Yet $|X| = v(G)\cdot\hat{N}(A)$ and $|Y|= v(G)\cdot\hat{N}(B)$. Moreover, 

$$|X\cap Y| = | V(G)\setminus \bigcup_{s\in A\cup B} N(s) | = v(G)\cdot\hat{N}(A\cup B).$$
 Because $X\cup Y \subseteq V(G)\setminus \bigcup_{s\in A\cap B}N(s)$, we see that
$$|X\cup Y| \leq |V(G)\setminus \bigcup_{s\in A\cap B}N(s)| =  v(G)\cdot\hat{N}(A\cap B),$$
and the proposition follows.
\end{proof}

We will also need the following lower bounds for $\hat{N}$ to ensure that certain factors in the objective function of the program are non-negative or even strictly positive.

\begin{proposition}\label{PositiveBound}
Let $G$ be a graph with $\delta(G) > \frac{3}{4}v(G)$. If $S\subseteq V(G)$ such that $|S|\le 4$, then

$$\hat{N}(S) > 1- \frac{|S|}{4} \ge 0.$$
\end{proposition}
\begin{proof}
Let $d_0 = 1 - \frac{\delta(G)}{v(G)}$. It follows from repeated applications of Proposition~\ref{IntersectionBound} that $\hat{N}(S) \ge 1 - |S|d_0$. Since $d_0 < 1/4$, we have that $\hat{N}(S) > 1-\frac{|S|}{4}$, which is at least $0$ since $|S|\le 4$.
\end{proof}

\begin{proposition}\label{PositiveBound2}
Let $G$ be a graph with $\delta(G) > \frac{3}{4}v(G)$. If $A,B\subseteq V(G)$ such that $|A\cup B|\le 4$, then

$$\hat{N}(A)+\hat{N}(B)-\hat{N}(A\cap B) > 0.$$
\end{proposition}
\begin{proof}
By Proposition~\ref{PositiveBound}, $\hat{N}(A) > 1 - \frac{|A|}{4}.$ By Proposition~\ref{IntersectionBound}, we have that 

$$\hat{N}(B) \ge \hat{N}(A\cap B) + \hat{N}(B\setminus A) - \hat{N}(\emptyset).$$

\noindent Recall that $\hat{N}(\emptyset)=1$. Furthermore, by Proposition~\ref{PositiveBound}, we have that  $\hat{N}(B\setminus A) \ge 1 - \frac{|B\setminus A|}{4}$. Hence

$$\hat{N}(B) - \hat{N}(A\cap B) \ge \hat{N}(B\setminus A) - 1 \ge -\frac{|B\setminus A|}{4}.$$

\noindent Combining this with our inequality for $\hat{N}(A)$, we find that

$$\hat{N}(A)+\hat{N}(B)-\hat{N}(A\cap B) > 1 - \frac{|A\cup B|}{4},$$

\noindent which is at least $0$ since $|A\cup B|\le 4$.
\end{proof}

\subsection{Main Program}\label{Program}

We now define a scaled version of $W$ as follows:

\begin{definition}
Let $G$ be a graph and let $r\in \{2,3,4\}$. For every $K=(v_1,\ldots, v_r)\in \cOK_r(G)$, we define a scaled weight 

\begin{align*}
\hat{W}(K) &= v(G)^{r-1} \cdot W(K) = \prod_{i=2}^r \frac{v(G)}{|\cK_{r+1}(G, \{v_1, \ldots, v_r\})|}= \prod_{i=2}^r \frac{1}{\hat{N}(\{v_1, \ldots, v_r\})}.
\end{align*}

\noindent For ease of reading, we will let $\hat{W}(v_1,\ldots, v_r) := \hat{W}(K)$.
\end{definition}

We may now rewrite $w_{G,1}(O)$ in terms of these scaled weights as follows. For ease of reading, we drop the set signs when taking $\hat{N}$ of a set of vertices.

\begin{proposition}\label{w1}
If $O=(x_1,x_2,x_3) \in \cOK_3(G)$ and $R=\bigcap_{i=1}^3 N(x_i)$, then

\begin{align*}
\hat{w}_{G,1}(O):= &\hat{N}(x_1,x_2) \cdot \frac{1}{v(G)} \cdot \sum_{y\in R} \Bigg( \hat{W}(x_1,y,x_2) - \hat{W}(x_1,x_2,y) \\
& + \frac{1}{v(G)} \cdot \sum_{z\in N(y)\cap R}\Bigg( \hat{W}(x_1,y,x_2,z) - \hat{W}(x_1,x_2,y,z) + \hat{W}(x_1,y,z,x_2) - \hat{W}(z,y,x_1,x_2)\Bigg) \Bigg).
\end{align*}
\end{proposition}

\noindent To prove Theorem~\ref{MainWThm}, it suffices to prove that the following program has value at most $1$.


\begin{table}[h]
  \begin{tabular}{llr@{}l}
      \multicolumn{1}{r}{{\rm (P1)}:} & \multicolumn{3}{l}{maximize $\hat{w}_{G,1}(O)$}  \\[0.2cm]
			\multicolumn{1}{r}{\text{s.t.}} & \multicolumn{3}{l}{$\text{ for all } y_i \in \bigcap_{k=1}^3N(x_k) \text{ and } z_{i,j}  \in N(y_i)\cap \bigcap_{k=1}^3N(x_k)$}\\[0.2cm]
     &\textbf{I. Degree constraints: } &$\hat{N}(x_1)$ & $\ \in [1-d,1], $ \\
 &  &$\hat{N}(y_i)$ &$\ \in [1-d,1], $ \\[0.2cm]
		&\textbf{II. Triangle constraints: }  &$\hat{N}(x_1,x_2)$&$\ \in [\hat{N}(x_1)-d,\hat{N}(x_1)]$,  \\
&  &$\hat{N}(x_1,y_i)$&$\ \in [\hat{N}(x_1)+\hat{N}(y_i)-1,1]$, \\
&  &$\hat{N}(y_i,z_{i,j})$ & $\ \in [\hat{N}(y_i)-d,\hat{N}(y_i)]$, \\[0.2cm]
&\textbf{III. $K_4$ constraints: }  &$\hat{N}(x_1,x_2,y_i)$ & $\ \in [\hat{N}(x_1,y_i)+\hat{N}(x_1,x_2)-\hat{N}(x_1),1]$, \\
&  &$\hat{N}(x_1,y_i,z_{i,j})$ &$\ \in [\hat{N}(x_1,y_i)+\hat{N}(y_i, z_{i,j})-\hat{N}(y_i),1]$, \\[0.2cm]
&\textbf{IV. $K_5$ constraints: }  &$\hat{N}(x_1,x_2,y_i,z_{i,j}) $& $\ \in [\hat{N}(x_1,x_2,y_i)+\hat{N}(y_i,z_{i,j})-\hat{N}(y_i),1]$. \\
  \end{tabular}
\end{table}

We note that the Degree Constraints follow from the bounds on the minimum degree. The Triangle Constraints for $\hat{N}(x_1,x_2)$ and $\hat{N}(y_i,z_{i,j})$ follow from Proposition~\ref{SubsetBound} while for $\hat{N}(x_1,y_i)$ they follow from Proposition~\ref{IntersectionBound}. Similarly, the $K_4$ Constraints and the $K_5$ Constraints follow from Proposition~\ref{IntersectionBound}. 

We also note that all of the variables are strictly positive by Proposition~\ref{PositiveBound} since $d < 0.25$. Moreover, each variable is at most $1$ since $\hat{N}(S) \le 1$ for every $S\subseteq V(G)$. Hence $\hat{w}_{G,1}(O)$ is well-defined and continuous in the domain of {\rm(P1)}.

Notice that we now think of these as variables. To make this more explicit,  let $R_0 = |\cK_4(G,O)|$ and for each $y_i\in \bigcap_{k=1}^3 N(x_k)$, we let $R_i = |\cK_5(G,V(O)\cup \{y_i\})|$. Let us replace the neighborhood densities above with variable names as follows:
\begin{itemize}
\item $\hat{N}(x_1) \rightarrow x$,
\item $\hat{N}(y_i) \rightarrow y_i$,
\item $\hat{N}(x_1,x_2) \rightarrow e_0$,
\item $\hat{N}(x_1,y_i) \rightarrow e_i$ for $i\in [R_0]$, 
\item $\hat{N}(y_i,z_{i,j}) \rightarrow f_{i,j}$ for $i\in [R_0], j\in [R_i]$,
\item $\hat{N}(x_1,x_2,y_i) \rightarrow q_{i,0}$ for $i\in [R_0]$,
\item $\hat{N}(x_1,y_i,z_{i,j}) \rightarrow q_{i,j}$ for $i\in [R_0], j\in [R_i]$,
\item $\hat{N}(x_1,x_2,y_i,z_{i,j}) \rightarrow p_{i,j}$ for $i\in [R_0], j\in [R_i]$.
\end{itemize}

Changing variable names as above in the formula from Proposition~\ref{w1}, we immediately see that $w_{G,1}(0)$ becomes:

$$\hat{W}_1 = \frac{e_0}{v(G)} \sum_{i=1}^{R_0} \Bigg( \frac{1}{q_{i,0}} \left(\frac{1}{e_i} - \frac{1}{e_0}\right) + \frac{1}{v(G)}\sum_{j=1}^{R_i} \Bigg( \frac{1}{p_{i,j}} \left( \frac{1}{q_{i,j}}\left(\frac{1}{e_i} - \frac{1}{f_{i,j}}\right) + \frac{1}{q_{i,0}}\left(\frac{1}{e_i} - \frac{1}{e_0}\right) \right) \Bigg) \Bigg).$$

The program then is as follows:


\begin{table}[h]
  \begin{tabular}{llr@{}l}
      \multicolumn{1}{r}{{\rm (P1)}:} & \multicolumn{3}{l}{maximize $\hat{W}_1$ } \\[0.2cm]
			\multicolumn{1}{r}{\text{s.t.}} & \multicolumn{3}{l}{$\text{ for all } i \in [R_0] \text{ and } j \in [R_i]$}\\[0.2cm]
     &\textbf{I. Degree constraints: } &$x$ & $\ \in [1-d,1], $ \\
&  &$y_i$ & $\ \in [1-d,1], $\\[0.2cm]
&\textbf{II. Triangle constraints: }  & $e_0$ & $\ \in [x-d,x], $  \\
&  &$e_i$ & $\ \in [x+y_i-1,1]$, \\
&  &$f_{i,j}$ & $\ \in [y_i-d,y_i]$, \\[0.2cm]
&\textbf{III. $K_4$ constraints: }  &$q_{i,0}$ & $\ \in [e_i+e_0-x,1]$, \\
&  &$q_{i,j}$ & $\ \in [e_i+f_{i,j}-y_i,1]$,  \\[0.2cm]
&\textbf{IV. $K_5$ constraints: }  &$p_{i,j}$ & $\ \in [q_{i,0}+f_{i,j}-y_i,1]$, \\[0.2cm]
&\textbf{V. Number of terms constraints: }  & $R_0$ & $\ \in [0,e_0 \cdot v(G)]$, \\
& &$R_i$ & $\ \in [0, q_{i,0} \cdot v(G)]$. \\
  \end{tabular}
\end{table}

Note that the bounds on $R_0$ and $R_i$ are derived as follows. First, $R_0 = v(G)\cdot \hat{N}(T)$. Yet $\hat{N}(T) \ge 0$ by Proposition~\ref{PositiveBound} and $\hat{N}(T) \le \hat{N}(e_0)$ (which has been renamed to just $e_0$) by Proposition~\ref{SubsetBound}. Similarly, $R_i = v(G) \cdot \hat{N}(T\cup \{y_i\})$. Yet $\hat{N}(T\cup \{y_i\}) \ge 0$ by Proposition~\ref{PositiveBound} and $\hat{N}(T\cup \{y_i\}) \le \hat{N}(e_0\cup e_i)$ (which has been renamed to $q_{i,0}$) by Proposition~\ref{SubsetBound}.

\subsection{Reduction to 10 Variables}\label{ReduceTo10}

Throughout this paper we will say that the maximum of an optimization program is \emph{achieved} to mean that there exists a point satisfying certain conditions so that the value of the objective function at that point is equal to the optimum value of the program; we note that the maximum may not necessarily be unique and may be achieved by points not satisfying these conditions.
\begin{lemma}\label{Lem1To2}
The maximum value of {\rm(P1)} is achieved when for all $i\in [R]$ and $j,j'\in [R_i]$, we have 

$$f_{i,j}=f_{i,j'}, q_{i,j}=q_{i,j'}, p_{i,j}=p_{i,j'}.$$ 
\end{lemma}
\begin{proof}
Since the domain of {\rm(P1)} is closed and bounded and $\hat{W_1}$ is well-defined and continuous on the domain of {\rm(P1)}, we find that {\rm(P1)} has a global maximum. Let $P_0$ be a point that achieves this maximum. For each $i$, let $j_i \in [R_i]$ such that 
$$\frac{1}{p_{i,j_i}} \left( \frac{1}{q_{i,j_i}}\left(\frac{1}{e_i} - \frac{1}{f_{i,j_i}}\right) + \frac{1}{q_{i,0}}\left(\frac{1}{e_i} - \frac{1}{e_0}\right) \right)$$
is maximized over all $j\in [R_i]$. Then the point $P_0'$ obtained from $P_0$ by setting $f_{i,j}=f_{i,j_i}, q_{i,j}=q_{i,j_i},$ and $p_{i,j}=p_{i,j_i}$ for all $i\in [R]$ and $j\in [R_i]$ is also a point that achieves this maximum. Moreover, since the constraints for the $f_{i,j},q_{i,j}, p_{i,j}$ are identical for each $j\in [R_i]$, it follows that $P_0'$ also satisfies the constraints of {\rm(P1)} as desired.
\end{proof}

Letting $r_i = \frac{R_i}{v(G)}$, we form a new program {\rm (P2)} with a new objective function that has the same optimum value as {\rm(P1)}:

$$\hat{W}_2 = \frac{e_0}{v(G)} \sum_{i=1}^{R_0} \Bigg( \frac{1}{q_{i,0}} \left(\frac{1}{e_i} - \frac{1}{e_0}\right) + r_i \Bigg( \frac{1}{p_i} \left( \frac{1}{q_i}\left(\frac{1}{e_i} - \frac{1}{f_i}\right) + \frac{1}{q_{i,0}}\left(\frac{1}{e_i} - \frac{1}{e_0}\right) \right) \Bigg) \Bigg).$$
%

\begin{table}[!h]
  \begin{tabular}{llr@{}l}
      \multicolumn{1}{r}{{\rm (P2)}:} & \multicolumn{3}{l}{maximize $\hat{W}_2$}  \\[0.2cm]
			\multicolumn{1}{r}{\text{s.t.}} & \multicolumn{3}{l}{$\text{ for all } i \in [R_0]$}\\[0.2cm]
     &\textbf{I. Degree constraints: } &$x$ & $\ \in [1-d,1], $ \\
		&  &$y_i$ & $\ \in [1-d,1], $ \\[0.2cm]
		&\textbf{II. Triangle constraints: }  &$e_0$& $\ \in [x-d,x]$,  \\
		&  &$e_i$ &$\ \in [x+y_i-1,1]$, \\
&  &$f_i$&$\ \in [y_i-d,y_i]$,  \\[0.2cm]
&\textbf{III. $K_4$ constraints: }  &$q_{i,0}$&$\ \in [e_i+e_0-x,1]$.\\
&  &$q_i$&$\ \in [e_i+f_i-y_i,1]$, \\[0.2cm]
&\textbf{IV. $K_5$ constraints: }  &$p_i$ & $\ \in [q_{i,0}+f_i-y_i,1]$,\\[0.2cm]
&\textbf{V. Number of terms constraints: }  &$R_0$&$\ \in [0,e_0\cdot v(G)]$,\\
& &$r_i$&$\ \in [0,q_{i,0}]$.\\
  \end{tabular}
\end{table}

\begin{corollary}\label{Cor1To2}
${\rm OPT}{\rm (P1)} = {\rm OPT}{\rm (P2)}$.
\end{corollary}
\begin{proof}
Follows from Lemma~\ref{Lem1To2}.
\end{proof}

We may now proceed to do the same with the $i$ variables as follows.

\begin{lemma}\label{Lem2To3}
The maximum of {\rm (P2)} is achieved when for all $i, i'\in [R]$, we have 

$$y_i= y_{i'},e_i=e_{i'},f_{i}=f_{i'}, q_{i,0}=q_{i',0},q_{i}=q_{i'}, p_{i}=p_{i'}.$$ 
\end{lemma}
\begin{proof}
Let $P_0$ be a point in the domain of $P_2$ that achieves the maximum of {\rm (P2)}. Let $I \in [R]$ such that 

$$\frac{1}{q_{i,0}} \left(\frac{1}{e_i} - \frac{1}{e_0}\right) + r_i \Bigg( \frac{1}{p_i} \left( \frac{1}{q_i}\left(\frac{1}{e_i} - \frac{1}{f_i}\right) + \frac{1}{q_{i,0}}\left(\frac{1}{e_i} - \frac{1}{e_0}\right) \right) \Bigg)$$

\noindent is maximized over all $i\in [R]$. Then the point $P_0'$ obtained from $P_0$ by setting $y_i = y_I, e_i= e_I, f_{i}=f_{I}, q_{i,0} = q_{I,0}, q_{i}=q_{I},$ and $p_{i}=p_{I}$ for all $i\in [R]$ is also a point that achieves this maximum. Moreover, since the constraints for the $y_i, e_i, f_i, q_{i,0},q_{i}, p_{i}$ are identical for each $i\in [R]$, it follows that $P_0'$ also satisfies the constraints of {\rm (P2)} as desired.
\end{proof}

Letting $r_0 = \frac{R_0}{v(G)}$, we form a new program {\rm (P3)} with a new objective function that has the same optimum value as {\rm (P2)} and hence as {\rm(P1)}:
$$\hat{W}_3(e_0,e,f,q_0,q,p,r_0,r) = e_0\cdot r_0\cdot \Bigg( \frac{1}{q_{0}} \left(\frac{1}{e} - \frac{1}{e_0}\right) + r \Bigg( \frac{1}{p} \left( \frac{1}{q}\left(\frac{1}{e} - \frac{1}{f}\right) + \frac{1}{q_{0}}\left(\frac{1}{e} - \frac{1}{e_0}\right) \right) \Bigg) \Bigg).$$
Here is the new program:

\begin{table}[!h]
  \begin{tabular}{llr@{}l}
      \multicolumn{1}{r}{{\rm (P3)}:} & \multicolumn{3}{l}{maximize $\hat{W}_3(e_0,e,f,q_0,q,p,r_0,r)$}  \\[0.2cm]
			\multicolumn{1}{r}{\text{s.t.}} & \multicolumn{3}{l}{$\ $}\\[0.2cm]
     &\textbf{I. Degree constraints: } &$x$ & $\ \in [1-d,1], $ \\
		&  &$y$ & $\ \in [1-d,1], $ \\[0.2cm]
		&\textbf{II. Triangle constraints: }  &$e_0$& $\ \in [x-d,x]$,  \\
		&  &$e$ &$\ \in [x+y-1,1]$, \\
&  &$f$&$\ \in [y-d,y]$,  \\[0.2cm]
&\textbf{III. $K_4$ constraints: }  &$q_{0}$&$\ \in [e+e_0-x,1]$.\\
&  &$q$&$\ \in [e+f-y,1]$, \\[0.2cm]
&\textbf{IV. $K_5$ constraints: }  &$p$ & $\ \in [q_{0}+f-y,1]$,\\[0.2cm]
&\textbf{V. Number of terms constraints: }  &$r_0$&$\ \in [0,e_0]$,\\
& &$r$&$\ \in [0,q_{0}]$.\\
  \end{tabular}
\end{table}

\begin{corollary}\label{Cor1To3}
{\rm OPT}{\rm (P3)} = {\rm OPT}{\rm(P1)}.
\end{corollary}
\begin{proof}
Follows from Lemma~\ref{Lem2To3} and Corollary~\ref{Cor1To2}.
\end{proof}

\section{Solving the Program}\label{Solving}

We do not actually solve {\rm (P3)}, rather in Subsection~\ref{ReduceTo6}, we upper bound {\rm (P3)} with a new program {\rm (P4)} that uses ramp functions (maximum of a function and 0) for each of the three main terms. This is the key that allows us to slowly reduce the number of variables, first to six {\rm (P5)}, then to five {\rm (P6)}, four {\rm (P7)}, three {\rm (P8)}, two {\rm (P9)}, one {\rm (P10)} and then we actually find the maximum point of {\rm (P10)}. Our value of $d$ in Theorem~\ref{RealMain} is precisely the maximum value of $d$ such that the value of this maximum point is 1 as required.

\subsection{Reduction to Six Variables}\label{ReduceTo6}

\begin{definition}
For a real-valued function $w(u)$ where $u\in \mathbb{R}^n$, define $w^+(u) = \frac{w(u)+|w(u)|}{2}$, that is the \emph{ramp function} of $w(u)$, i.e.~taking $w(u)$ if the function has positive value and $0$ otherwise.
\end{definition}

We note the following basic fact about ramp functions. 

\begin{proposition}
If $w$ is a real-valued function that is continuous on a region $R$, then $w^+$ is continuous on $R$.
\end{proposition}

We now construct a new program {\rm (P4)} whose minimum is at most that of {\rm (P3)} by retaining the same constraints but changing the objective function. Namely, we replace each cancellation in $W_3$ by its ramp function as follows:

\begin{align*}
\hat{W}_4(e_0,e,f,q_0,q,p,r_0,r) &= e_0\cdot r_0\cdot \Bigg( \frac{(e_0-e)^+}{q_{0}ee_0}  + r\frac{(f-e)^+}{pqef} + r\frac{(e_0-e)^+}{pq_{0}ee_0} \Bigg)\\
&= \frac{r_0(e_0-e)^+}{q_0e} + \frac{e_0r_0r(f-e)^+}{pqef} + \frac{e_0r_0r(e_0-e)^+}{pq_0ee_0}.
\end{align*}

\begin{lemma}\label{Lem3To4}
$\hat{W}_4 \ge 0$ and $\hat{W}_4 \ge \hat{W}_3$ in the domain of {\rm (P4)}.
\end{lemma}
\begin{proof}
This follows since $e_0,e,f,q_0,q,p>0$ and $r_0,r \ge 0$ in the domain of {\rm (P4)}.
\end{proof}

\begin{corollary}\label{Cor1To4}
${\rm OPT}{\rm (P4)}\ge {\rm OPT}{\rm (P3)} = {\rm OPT}{\rm (P1)}.$
\end{corollary}

\begin{lemma}\label{Lem4To5}
The maximum of {\rm (P4)} is achieved when all of the following hold:
\begin{itemize}
\item $r=q_0$,
\item $r_0=e_0$,
\item $p=q_0+f-y$,
\item $q=e+f-y$.
\end{itemize}
\end{lemma}
\begin{proof}
Let $P_0=(x,y,e_0,e,f,q_0,q,p,r_0,r)$ be a point that achieves the maximum of {\rm (P4)}. Let $r'=q_0, r_0'=e_0, p'=q_0+f-y, q'= e+f-y$. Let $P_0' = (x,y,e_0,e,f,q_0,q',p',r_0',r')$. Note that $P_0'$ is in the domain of {\rm (P4)} since none of $r_0,r,p,$ or $q$ appear in the constraints of other variables in {\rm (P4)}. 

It suffices to prove that $\hat{W}_4(P_0') \ge \hat{W}_4(P_0)$. Since $P_0$ is in the domain of {\rm (P4)}, we have that $r\le r', r_0\le r_0', p\ge p', q\ge q'$. Since $P_0$ is in the domain of {\rm (P4)}, we have by Proposition~\ref{PositiveBound} that $x,y,e_0,e,f,q_0,q,q > 0$ and $r_0,r \ge 0$. Hence all the factors in denominators in $\hat{W_4}$ are strictly positive while all factors in numerators in $\hat{W_4}$ are non-negative. But then
$$\frac{r_0(e_0-e)^+}{q_0e} \le \frac{r_0'(e_0-e)^+}{q_0e},$$
\noindent and
$$\frac{e_0r_0r(f-e)^+}{pqef} \le \frac{e_0r_0'r'(f-e)^+}{p'q'ef},$$
\noindent and
$$\frac{e_0r_0r(e_0-e)^+}{pq_0ee_0} \le \frac{e_0r_0'r'(e_0-e)^+}{p'q_0ee_0}.$$
\noindent Hence $\hat{W}_4(P_0') \ge \hat{W}_4(P_0)$ as desired.
\end{proof}

Thus we construct a new program {\rm (P5)} with the following new objective function and a subset of the previous constraints as follows:
\begin{align*}
\hat{W}_5(y,e_0,e,f,q_0) &= \frac{e_0(e_0-e)^+}{q_{0}e}  +  \frac{e_0^2 q_0(f-e)^+}{(q_0+f-y)(e+f-y)ef} + \frac{e_0(e_0-e)^+}{(q_0+f-y)e}
\end{align*}

Here is the new program:

\begin{table}[!h]
  \begin{tabular}{llr@{}l}
      \multicolumn{1}{r}{{\rm (P5)}:} & \multicolumn{3}{l}{maximize $\hat{W}_5(y,e_0,e,f,q_0)$}  \\[0.2cm]
			\multicolumn{1}{r}{\text{s.t.}} & \multicolumn{3}{l}{$\ $}\\[0.2cm]
     &\textbf{I. Degree constraints: } &$x$ & $\ \in [1-d,1], $ \\
		&  &$y$ & $\ \in [1-d,1], $ \\[0.2cm]
		&\textbf{II. Triangle constraints: }  &$e_0$& $\ \in [x-d,x]$,  \\
		&  &$e$ &$\ \in [x+y-1,1]$, \\
&  &$f$&$\ \in [y-d,y]$,  \\[0.2cm]
&\textbf{III. $K_4$ constraints: }  &$q_{0}$&$\ \in [e+e_0-x,1]$.\\
  \end{tabular}
\end{table}

\newpage
\begin{corollary}
${\rm OPT}{\rm (P5)} \ge {\rm OPT}{\rm (P1)}.$
\end{corollary}

\subsection{Reduction to Four Variables}\label{ReduceTo4}

We proceed with reducing $q_0$ as follows.
\begin{lemma}\label{Lem5To6}
The maximum of {\rm (P5)} is achieved when $q_0=e+e_0-x$. 
\end{lemma}
\begin{proof}
Let $P_0 = (x,y,e_0,e,f,q_0)$ be a point that achieves the maximum of {\rm (P5)}. Let $q_0'=e+e_0-x$. Let $P_0' = (x,y,e_0,e,f,q_0')$. Note that $P_0'$ is in the domain of {\rm (P5)} since $q_0$ does not appear in the constraints of other variables in {\rm (P5)}. 

It suffices to prove that $\hat{W}_5(P_0') \ge \hat{W}_5(P_0)$. Since $P_0$ is in the domain of {\rm (P5)}, we have that $q_0\ge q_0'$. Since $P_0$ is in the domain of {\rm (P5)}, we have by Proposition~\ref{PositiveBound} that $x,y,e_0,e,f,q_0 > 0$ and also by Proposition~\ref{PositiveBound2} that $q_0+f-y, e+f-y>0$. But then
$$\frac{e_0(e_0-e)^+}{q_{0}e} \le   \frac{e_0(e_0-e)^+}{q_{0}'e},$$
\noindent and
$$\frac{e_0(e_0-e)^+}{(q_0+f-y)e}\le \frac{e_0(e_0-e)^+}{(q_0'+f-y)e}.$$
\noindent Since $f-y \le 0$ and $q_0'+f-y > 0$, it follows that $(f-y)q_0 \le (f-y)q_0'$ and hence
$$\frac{q_0}{q_0+f-y} \le \frac{q_0'}{q_0'+f-y}.$$
\noindent Thus,
$$\frac{e_0^2 q_0(f-e)^+}{(q_0+f-y)(e+f-y)ef} \le \frac{e_0^2 q_0'(f-e)^+}{(q_0'+f-y)(e+f-y)ef}.$$
\noindent Hence $\hat{W}_5(P_0') \ge \hat{W}_5(P_0)$ as desired.
\end{proof}

Thus we may replace {\rm (P5)} with a new program {\rm (P6)} whose minimum value is at most that of {\rm (P5)} by setting $q_0=e+e_0-x$ as follows:
\begin{align*}
\hat{W}_6(x,y,e_0,e,f) &= \frac{e_0(e_0-e)^+}{(e+e_0-x)e} + \frac{e_0^2 (e+e_0-x)(f-e)^+}{(e+e_0-x+f-y)(e+f-y)ef} + \frac{e_0(e_0-e)^+}{(e+e_0-x+f-y)e}.
\end{align*}

Here is the new program:

\begin{table}[!h]
  \begin{tabular}{llr@{}l}
      \multicolumn{1}{r}{{\rm (P6)}:} & \multicolumn{3}{l}{maximize $\hat{W}_6(x,y,e_0,e,f)$}  \\[0.2cm]
			\multicolumn{1}{r}{\text{s.t.}} & \multicolumn{3}{l}{$\ $}\\[0.2cm]
     &\textbf{I. Degree constraints: } &$x$ & $\ \in [1-d,1], $ \\
		&  &$y$ & $\ \in [1-d,1], $ \\[0.2cm]
		&\textbf{II. Triangle constraints: }  &$e_0$& $\ \in [x-d,x]$,  \\
		&  &$e$ &$\ \in [x+y-1,1]$, \\
&  &$f$&$\ \in [y-d,y]$.  \\[0.2cm]
  \end{tabular}
\end{table}

\newpage
\begin{corollary}
${\rm OPT}{\rm (P6)} \ge {\rm OPT}{\rm (P1)}.$
\end{corollary}

We now proceed with reducing $e$ as follows.
\begin{lemma}\label{Lem6To7}
The maximum of {\rm (P6)} is achieved when $e=x+y-1$. 
\end{lemma}
\begin{proof}
Let $P_0 = (x,y,e_0,e,f)$ be a point that achieves the maximum of {\rm (P6)}. Let $e'=x+y-1$. Let $P_0' = (x,y,e_0,e',f)$. Note that $P_0'$ is in the domain of {\rm (P6)} since $e$ does not appear in the constraints of other variables in {\rm (P6)}. 

It suffices to prove that $\hat{W}_6(P_0') \ge \hat{W}_6(P_0)$. Since $P_0$ is in the domain of {\rm (P6)}, we have that $e\ge e'$. Since $P_0$ is in the domain of {\rm (P6)}, we have by Proposition~\ref{PositiveBound} that $x,y,e_0,e,f,e' > 0$ and also by Proposition~\ref{PositiveBound2} that $e+e_0-x+f-y\geq e'+e_0-x+f-y > 0$ and $e+f-y \geq e'+f-y > 0$.

\begin{claim}\label{EFirst}
$$\frac{e_0(e_0-e)^+}{(e+e_0-x)e} \le  \frac{e_0(e_0-e')^+}{(e'+e_0-x)e'}.$$
\end{claim}
\begin{proof}
If $e \ge e_0$, then $(e_0-e)^+=0$ and the claim follows. So we may assume that $e < e_0$ and hence $e' < e_0$. But then
$$\frac{e_0-e}{(e+e_0-x)e} \le \frac{e_0-e'}{(e'+e_0-x)e'},$$
\noindent and the result follows by multiplying the above inequality by $e_0$.
\end{proof}

\begin{claim}\label{ESecond}
$$\frac{e_0^2 (e+e_0-x)(f-e)^+}{(e+e_0-x+f-y)(e+f-y)ef} \le \frac{e_0^2 (e'+e_0-x)(f-e')^+}{(e'+e_0-x+f-y)(e'+f-y)e'f}.$$
\end{claim}
\begin{proof}
If $e\ge f$, then $(f-e)^+=0$ and the claim follows. So we may assume that $e<f$ and hence $e'<f$. Since $f-y\le 0$ and $e \ge e' > 0$, we have that $(f-y)e \le (f-y)e'$. Hence we have
$$\frac{e+e_0-x}{e+e_0-x+f-y} \le \frac{e'+e_0-x}{e'+e_0-x+f-y},$$
since all of the terms in the inequality are strictly positive as noted above and cross multiplying and canceling like terms yields $(f-y)e \le (f-y)e'$.
\noindent Moreover we have
$$\frac{f-e}{(e+f-y)e} \le \frac{f-e'}{(e'+f-y)e'}$$
because as noted above all terms in the inequality are strictly positive and $e \ge e'$.  Multiplying the two above inequalities (whose left sides are both strictly positive) and then multiplying by $\frac{e_0^2}{f}$ (which is also positive) gives the desired inequality.
\end{proof}

\begin{claim}\label{EThird}
$$\frac{e_0(e_0-e)^+}{(e+e_0-x+f-y)e} \le \frac{e_0(e_0-e')^+}{(e'+e_0-x+f-y)e'}.$$
\end{claim}
\begin{proof}
If $e \ge e_0$, then $(e_0-e)^+=0$ and the claim follows. So we may assume that $e < e_0$ and hence $e' < e_0$. But then
$$\frac{e_0-e}{(e+e_0-x+f-y)e} \le \frac{e_0-e'}{(e'+e_0-x+f-y)e'},$$
\noindent and the result follows by multiplying the above inequality by $e_0$.
\end{proof}
 \noindent It follows from Claims~\ref{EFirst},~\ref{ESecond}, and~\ref{EThird} that $\hat{W}_6(P_0')\ge \hat{W}_6(P_0)$ as desired.
\end{proof}

Thus we may replace {\rm (P6)} with a new program {\rm (P7)} whose maximum value is at least that of {\rm (P6)} by setting $e=x+y-1$ as follows:
\begin{align*}
\hat{W}_7(x,y,e_0,f) = &\frac{e_0(e_0-x-y+1)^+}{(e_0+y-1)(x+y-1)} + \frac{e_0^2 (y-1+e_0)(f-x-y+1)^+}{(e_0+f-1)(x-1+f)(x+y-1)f} \\
&+ \frac{e_0(e_0-x-y+1)^+}{(e_0+f-1)(x+y-1)}
\end{align*}

Here is the new program:

\begin{table}[!h]
  \begin{tabular}{llr@{}l}
      \multicolumn{1}{r}{{\rm (P7)}:} & \multicolumn{3}{l}{maximize $\hat{W}_7(x,y,e_0,f)$}  \\[0.2cm]
			\multicolumn{1}{r}{\text{s.t.}} & \multicolumn{3}{l}{$\ $}\\[0.2cm]
     &\textbf{I. Degree constraints: } &$x$ & $\ \in [1-d,1], $ \\
		&  &$y$ & $\ \in [1-d,1], $ \\[0.2cm]
		&\textbf{II. Triangle constraints: }  &$e_0$& $\ \in [x-d,x]$,  \\
&  &$f$&$\ \in [y-d,y]$.  \\[0.2cm]
  \end{tabular}
\end{table}

\begin{corollary}
${\rm OPT}{\rm (P7)} \ge {\rm OPT}{\rm (P1)}.$
\end{corollary}

\subsection{Reduction to Two Variables}\label{ReduceTo2}

Before proceeding, it will be useful to switch the variables. To that end we introduce two new variables $a,b$ to replace $f,e_0$ respectively as follows:
\begin{itemize}
\item $e_0 = x - a$,
\item $f = y - b$.
\end{itemize}

Here then is program {\rm (P7)} with these new variables and constraints:
\begin{align*}
\hat{W}_7(x,y,a,b) = &\frac{(x-a)(1-y-a)^+}{(x+y-1-a)(x+y-1)} +\frac{(x-a)^2 (x+y-1-a)(1-x-b)^+}{(x+y-1-a-b)(x+y-1-b)(x+y-1)(y-b)} \\
&+ \frac{(x-a)(1-y-a)^+}{(x+y-1-a-b)(x+y-1)}
\end{align*}

Here is the new program:

\begin{table}[!h]
  \begin{tabular}{llr@{}l}
      \multicolumn{1}{r}{{\rm (P7)}:} & \multicolumn{3}{l}{maximize $\hat{W}_7(x,y,a,b)$}  \\[0.2cm]
			\multicolumn{1}{r}{\text{s.t.}} & \multicolumn{3}{l}{$\ $}\\[0.2cm]
     &\textbf{I. Degree constraints: } &$x$ & $\ \in [1-d,1], $ \\
		&  &$y$ & $\ \in [1-d,1], $ \\[0.2cm]
		&\textbf{II. Triangle constraints: }  &$a$& $\ \in [0,d]$,  \\
&  &$b$&$\ \in [0,d]$.  \\[0.2cm]
  \end{tabular}
\end{table}

\newpage

We now proceed with reducing $y$ as follows.
\begin{lemma}\label{Lem7To8}
The maximum of {\rm (P7)} is achieved when $y=1-d$. 
\end{lemma}
\begin{proof}
Let $P_0 = (x,y,a,b)$ be a point that achieves the maximum of {\rm (P7)}. Let $y'=1-d$. Let $P_0' = (x,y',a,b)$. Note that $P_0'$ is in the domain of {\rm (P7)} since $y$ does not appear in the constraints of other variables in {\rm (P7)}. 

It suffices to prove that $\hat{W}_7(P_0') \ge \hat{W}_7(P_0)$. Since $P_0$ is in the domain of {\rm (P7)}, we have that $y\ge y'$. Since $P_0$ is in the domain of {\rm (P7)}, we have that $x,y> 0$, $a,b\geq 0$, and also that $x+y'-1-a-b = x-d-a-b > 0$ since $x\ge 1-d, a,b\le d$ and $ d < 1/4$.

\begin{claim}\label{YFirst}
$$\frac{(x-a)(1-y-a)^+}{(x+y-1-a)(x+y-1)} \le  \frac{(x-a)(1-y'-a)^+}{(x+y'-1-a)(x+y'-1)}.$$
\end{claim}
\begin{proof}
If $y \ge 1-a$, then $(1-y-a)^+=0$ and the claim follows. So we may assume that $y < 1-a$ and hence $y' < 1-a$. But then
$$\frac{1-y-a}{(x+y-1-a)(x+y-1)} \le \frac{1-y'-a}{(x+y'-1-a)(x+y'-1)},$$
\noindent and the result follows by multiplying the above inequality by $x-a$.
\end{proof}

\begin{claim}\label{YSecond}
\begin{align*}
&\frac{(x-a)^2 (x+y-1-a)(1-x-b)^+}{(x+y-1-a-b)(x+y-1-b)(x+y-1)(y-b)} \\
\le &\frac{(x-a)^2 (x+y'-1-a)(1-x-b)^+}{(x+y'-1-a-b)(x+y'-1-b)(x+y'-1)(y'-b)}.
\end{align*}
\end{claim}
\begin{proof}
If $x\ge 1-b$, then $(1-x-b)^+=0$ and the claim follows. So we may assume that $x <1-b$. Since $b\ge 0$ and $y \ge y' > 0$, we have that $-by \le -by'$. Since $x+y'-1-a > 0$, we have that
$$\frac{x+y-1-a}{x+y-1-a-b} \le \frac{x+y'-1-a}{x+y'-1-a-b},$$
cross multiplying and canceling like terms yields $-by \le -by'$. Moreover we also have that
$$\frac{1}{(x+y-1-b)(x+y-1)(y-b)} \le \frac{1}{(x+y'-1-b)(x+y'-1)(y'-b)}.$$
Multiplying the two above inequalities (whose left sides are both strictly positive) and then multiplying by $(x-a)^2(1-x-b)$ (which is also positive) gives the desired inequality.
\end{proof}

\begin{claim}\label{YThird}
$$\frac{(x-a)(1-y-a)^+}{(x+y-1-a-b)(x+y-1)} \le \frac{(x-a)(1-y'-a)^+}{(x+y'-1-a-b)(x+y'-1)}.$$
\end{claim}
\begin{proof}
If $y \ge 1-a$, then $(1-y-a)^+=0$ and the claim follows. So we may assume that $y < 1-a$ and hence $y' < 1-a$. But then
$$\frac{1-y-a}{(x+y-1-a-b)(x+y-1)} \le \frac{1-y'-a}{(x+y'-1-a-b)(x+y'-1)},$$
\noindent and the result follows by multiplying the above inequality by $x-a$.
\end{proof}
\noindent It follows from Claims~\ref{YFirst},~\ref{YSecond}, and~\ref{YThird} that $\hat{W}_7(P_0')\ge \hat{W}_7(P_0)$ as desired.
\end{proof}

Thus we may replace {\rm (P7)} with a new program {\rm (P8)} whose maximum value is at least that of {\rm (P7)} by setting $y=1-d$ as follows:
\begin{align*}
\hat{W}_8(x,a,b) = &\frac{(x-a)(d-a)^+}{(x-d-a)(x-d)} +\frac{(x-a)^2 (x-d-a)(1-x-b)^+}{(x-d-a-b)(x-d-b)(x-d)(1-d-b)} \\
&+ \frac{(x-a)(d-a)^+}{(x-d-a-b)(x-d)}
\end{align*}

Here is the new program:
\begin{table}[!h]
  \begin{tabular}{llr@{}l}
      \multicolumn{1}{r}{{\rm (P8)}:} & \multicolumn{3}{l}{maximize $\hat{W}_8(x,a,b)$}  \\[0.2cm]
			\multicolumn{1}{r}{\text{s.t.}} & \multicolumn{3}{l}{$\ $}\\[0.2cm]
     &\textbf{I. Degree constraints: } &$x$ & $\ \in [1-d,1], $ \\[0.2cm]
		&\textbf{II. Triangle constraints: }  &$a,b$& $\ \in [0,d]$.  \\
  \end{tabular}
\end{table}
\begin{corollary}
${\rm OPT}{\rm (P8)} \ge {\rm OPT}{\rm (P1)}.$
\end{corollary}

We now proceed with reducing $x$ as follows.

\begin{lemma}\label{Lem8To9}
The maximum of {\rm (P8)} is achieved when $x=1-d$. 
\end{lemma}
\begin{proof}
Let $P_0 = (x,a,b)$ be a point that achieves the maximum of {\rm (P8)}. Let $x'=1-d$. Let $P_0' = (x',a,b)$. Note that $P_0'$ is in the domain of {\rm (P8)} since $x$ does not appear in the constraints of other variables in {\rm (P8)}. 

It suffices to prove that $\hat{W}_8(P_0') \ge \hat{W}_8(P_0)$. Since $P_0$ is in the domain of {\rm (P8)}, we have that $x\ge x'$. Since $P_0$ is in the domain of {\rm (P8)}, we have that $a,b \geq 0$,  $x> 0$, and also that $x-d-a > x'-d-a = 1-2d-a > 0$ since $x\ge 1-d, a\le d$ and $ d < 1/3$.

\begin{claim}\label{XFirst}
$$\frac{(x-a)(d-a)^+}{(x-d-a)(x-d)} \le  \frac{(x'-a)(d-a)^+}{(x'-d-a)(x'-d)}.$$
\end{claim}
\begin{proof}
Since $d\ge 0$, we have that $-dx \le -dx'$. Since $x-d-a\ge x'-d-a > 0$, we have
$$\frac{x-a}{x-d-a} \le \frac{x'-a}{x'-d-a},$$
since all of the terms in the inequality are strictly positive as noted above and cross multiplying and canceling like terms yields $-dx \le -dx'$.  Moreover, $\frac{1}{x-d} \le \frac{1}{x'-d}$. Multiplying these two inequalities together (whose left sides are both positive) and then multiplying by $(d-a)^+$ (which is non-negative) gives the desired inequality.
\end{proof}

\begin{claim}\label{XSecond}
$$\frac{(x-a)^2 (x-d-a)(1-x-b)^+}{(x-d-a-b)(x-d-b)(x-d)(1-d-b)} 
\le \frac{(x'-a)^2 (x'-d-a)(1-x'-b)^+}{(x'-d-a-b)(x'-d-b)(x'-d)(1-d-b)}.$$
\end{claim}
\begin{proof}
If $x\ge 1-b$, then $(1-x-b)^+=0$ and the claim follows. So we may assume that $x <1-b$ and hence $x'<1-b$. 

Since $a\ge d$ and $x \ge x' > 0$, we have that $-(d-a)x \le -(d-a)x'$. Since $x-d\ge x'-d > 0$, we have that
$$\frac{x-a}{x-d} \le \frac{x'-a}{x'-d}.$$
\noindent Similarly since $a \le d+b$ and $x'-d-b > 0$, we have that
$$\frac{x-a}{x-d-b} \le \frac{x'-a}{x'-d-b}.$$
\noindent Similarly since $b\ge 0$ and $x'-d-a-b > 0$, we have that
$$\frac{x-d-a}{x-d-a-b} \le \frac{x'-d-a}{x'-d-a-b}.$$
Finally, we note that $1-x-b \le 1-x'-b$. Multiplying the four above inequalities (whose left sides are strictly positive) and then multiplying by $\frac{1}{1-d-b}$ (which is also positive) gives the desired inequality.
\end{proof}

\begin{claim}\label{XThird}
$$\frac{(x-a)(d-a)^+}{(x-d-a-b)(x-d)} \le \frac{(x'-a)(d-a)^+}{(x'-d-a-b)(x'-d)}.$$
\end{claim}
\begin{proof}
Since $d,b\ge 0$, we have that $-(d+b)x \le -(d+b)x'$. Since $x-d-a-b\ge x'-d-a-b > 0$, we have
$$\frac{x-a}{x-d-a-b} \le \frac{x'-a}{x'-d-a-b}.$$
Moreover, $\frac{1}{x-d} \le \frac{1}{x'-d}$. Multiplying these two inequalities together (whose left sides are both positive) and then multiplying by $(d-a)^+$ (which is non-negative) gives the desired inequality.
\end{proof}
It follows from Claims~\ref{XFirst},~\ref{XSecond}, and~\ref{XThird} that $\hat{W}_8(P_0')\ge \hat{W}_8(P_0)$ as desired.
\end{proof}

Thus we may replace {\rm (P8)} with a new program {\rm (P9)} whose maximum value is at least that of {\rm (P8)} by setting $x=1-d$ as follows:
\begin{align*}
\hat{W}_9(a,b) = &\frac{(1-d-a)(d-a)}{(1-2d-a)(1-2d)} +\frac{(1-d-a)^2 (1-2d-a)(d-b)}{(1-2d-a-b)(1-2d-b)(1-2d)(1-d-b)} \\
&+ \frac{(1-d-a)(d-a)}{(1-2d-a-b)(1-2d)}
\end{align*}
Note that we have dropped the ramp functions at this point since $a,b\le d$ and hence $d-a,d-b\ge 0$.

Here is the new program:

\begin{table}[h]
  \begin{tabular}{llr@{}l}
      \multicolumn{1}{r}{{\rm (P9)}:} & \multicolumn{3}{l}{maximize $\hat{W}_9(a,b)$}  \\[0.2cm]
			\multicolumn{1}{r}{\text{s.t.}} & \multicolumn{3}{l}{$a,b\in [0,d].$}\\[0.2cm] \end{tabular}
\end{table}

\begin{corollary}
${\rm OPT}{\rm (P9)} \ge {\rm OPT}{\rm (P1)}.$
\end{corollary}

\subsection{Reduction to One Variable}\label{ReduceTo1}

We now proceed with reducing $a$ as follows.
\begin{lemma}\label{Lem9To10}
The maximum of {\rm (P9)} is achieved when $a=0$. 
\end{lemma}
\begin{proof}
Let $P_0 = (a,b)$ be a point that achieves the maximum of {\rm (P9)}. Let $P_0' = (0,b)$. Note that $P_0'$ is in the domain of {\rm (P9)}.

It suffices to prove that $\hat{W}_9(P_0') \ge \hat{W}_9(P_0)$. 
\begin{claim}\label{AFirst}
$$\frac{(1-d-a)(d-a)}{(1-2d-a)(1-2d)} \le  \frac{(1-d)d}{(1-2d)(1-2d)}.$$
\end{claim}
\begin{proof}
Since $a\ge 0$, we have that $1-d-a \le 1-d$. Moreover, $d \le 1-2d$ since $d\le 1/3$. Since $1-2d-a> 0$, it follows that
$$\frac{d-a}{1-2d-a} \le \frac{d}{1-2d}.$$
\noindent Multiplying these two inequalities together (whose left sides are both positive) and then multiplying by $\frac{1}{1-2d}$ (which is non-negative) gives the desired inequality.
\end{proof}
\begin{claim}\label{ASecond}
$$\frac{(1-d-a)^2 (1-2d-a)}{1-2d-a-b} \le \frac{(1-d)^2 (1-2d)}{1-2d-b}.$$
\end{claim}
\begin{proof}
Let $H(s,t) = \frac{(1-d-s)^2 (1-2d-s)}{1-2d-s-t}$. It suffices to prove that $\frac{\partial H}{\partial s}(s,t) \ge 0$ for all $s,t\in [0,d]$.
Note that
\begin{align*}
\frac{\partial H}{\partial s}(s,t) &= H(s,t) \cdot \left(\frac{2}{1-d-s} + \frac{1}{1-2d-s} - \frac{1}{1-2d-s-t} \right) \\
&=H(s,t) \cdot \left(\frac{2}{1-d-s} - \frac{t}{(1-2d-s)(1-2d-s-t)} \right)\\
&=\frac{H(s,t)}{1-d-s} \cdot \left(2 - \frac{t(1-d-s)}{(1-2d-s)(1-2d-s-t)} \right).
\end{align*}
\noindent Let 
$$H_1(s,t) = \frac{t(1-d-s)}{(1-2d-s)(1-2d-s-t)}.$$ 
\noindent Recall that $H(s,t)\ge 0$ and $1-d-s\ge 0$ for all $s,t\in [0,d]$ since $d\le 1/4$. Thus it suffices to show $H_1(s,t)\le 2$.

Note that $H_1(s,t) \le H_1(s,d)$ since $t\le d$ and $\frac{1}{1-2d-s-t} \le \frac{1}{1-3d-s}$. We claim that $H_1(s,d) \le H_1(d,d)$. To see this, note that $\frac{1}{1-3d-s} \le \frac{1}{1-4d}$ since $s\le d$ and $d<1/4$ and similarly that
$$\frac{1-d-s}{1-2d-s} \le \frac{1-2d}{1-3d}.$$
\noindent This proves the claim that 
$$H_1(s,d)\le H_1(d,d)  = \frac{d(1-2d)}{(1-3d)(1-4d)}.$$
\noindent But then $H_1(d,d) \le 2$ when $d(1-2d) \le 2(1-3d)(1-4d)$ or equivalently when 
$$26d^2-15d+2 \ge 0.$$ The roots of this quadratic equation are $\frac{15-\sqrt{17}}{52} \approx 0.20917$ and $\frac{15+\sqrt{17}}{52}\approx 0.36775$. Since $d\le 1/5$, it follows that $26d^2-15d+2 \ge 0$ and hence $H_1(s,t)\le H_1(d,d)\le 2$ as desired.
\end{proof}
\begin{claim}\label{AThird}
$$\frac{(1-d-a)(d-a)}{(1-2d-a-b)(1-2d)} \le  \frac{(1-d)d}{(1-2d-b)(1-2d)}.$$
\end{claim}
\begin{proof}
Since $a\ge 0$, we have that $1-d-a \le 1-d$. Moreover, $d \le 1-2d-b$ since $b\le d$ and $d\le 1/4$. Since $1-2d-a-b\ge 0$, it follows that
$$\frac{d-a}{1-2d-a-b} \le \frac{d}{1-2d-b}.$$
\noindent Multiplying these two inequalities together (whose left sides are both positive) and then multiplying by $\frac{1}{1-2d}$ (which is non-negative) gives the desired inequality.
\end{proof}
\noindent It follows from Claims~\ref{AFirst},~\ref{ASecond}, and~\ref{AThird} that $\hat{W}_9(P_0')\ge \hat{W}_9(P_0)$ as desired.
\end{proof}

Thus we may replace {\rm (P8)} with a new program {\rm (P9)} whose maximum value is at least that of {\rm (P8)} by setting $a=0$ as follows:
\begin{align*}
\hat{W}_{10}(b) = &\frac{(1-d)d}{(1-2d)^2} +\frac{(1-d)^2(d-b)}{(1-2d-b)^2(1-d-b)} + \frac{(1-d)d}{(1-2d-b)(1-2d)}.
\end{align*}
Here is the new program:

\begin{table}[h]
  \begin{tabular}{llr@{}l}
      \multicolumn{1}{r}{{\rm (P10)}:} & \multicolumn{3}{l}{maximize $\hat{W}_{10}(b)$}  \\[0.2cm]
			\multicolumn{1}{r}{\text{s.t.}} & \multicolumn{3}{l}{$b\in [0,d]$.}\\[0.2cm] \end{tabular}
\end{table}

\begin{corollary}
${\rm OPT}{\rm (P10)} \ge {\rm OPT}{\rm (P1)}.$
\end{corollary}

\subsection{The Final Optimization}\label{FinalOpt}

We now proceed with reducing $b$ as follows.

\begin{lemma}\label{Lem10To11}
The maximum of {\rm (P10)} is achieved when $b=0$. 
\end{lemma}
\begin{proof}
Note that $\hat{W}_{10}(b)= \frac{(1-d)d}{(1-2d)^2} + \frac{1-d}{1-2d}G(b)$ where
\begin{align*}
G(b) &= \frac{(1-d)(1-2d)(d-b)}{(1-2d-b)^2(1-d-b)} + \frac{d}{1-2d-b}\\[0.2cm]
&= \frac{(1-d)(d-b)(1-2d) + d(1-2d-b)(1-d-b)}{(1-2d-b)^2(1-d-b)}\\[0.2cm]
&= \frac{2d(1-d)(1-2d) + b(-1+d+d^2) + b^2d}{(1-2d-b)^2(1-d-b)}
\end{align*}
\noindent It suffices to show that $G(b) \le G(0)$ for all $b\in [0,d]$ where
$$G(0) = \frac{2d(1-d)(1-2d)}{(1-2d)^2(1-d)} = \frac{2d}{1-2d}.$$
\noindent Now 
$$G(b)-G(0) = \frac{F(b)}{(1-2d-b)^2(1-d-b)(1-2d)}$$
\noindent where
\begin{align*}
F(b) &= (1-2d)(2d(1-d)(1-2d) + b(-1+d+d^2)+b^2d) - 2d(1-2d-b)^2(1-d-b)\\[0.2cm]
&= b\Big((1-2d)(-1+d+d^2)+2d((1-d)^2+ 2(1-d)(1-2d))\Big)\\[0.2cm]
&\ \ \ \ + b^2 \Big((1-2d)d-2d(3-4d)\Big) + b^3(2d)\\[0.2cm]
&= b\Big(-1-d+3d^2-2d^3+2d(1-d)(3-5d)\Big) + b^2\Big(-5d+10d^2\Big)+b^3(2d)\\[0.2cm]
&= b\Big((-1+5d-13d^2-12d^3)+ b(-5d+10d^2)+b^2(2d)\Big).\\[0.2cm]
\end{align*}
\noindent Since $b\ge 0$, it suffices to show that $E(b)=F(b)/b \le 0$, that is it suffices to show that
$$E(b)=(-1+5d-13d^2-12d^3)+ b(-5d+10d^2)+b^2(2d) \le 0,$$
\noindent for all $b\in [0,d]$. 
To that end, we claim that $E(b)\le E(0)$. Note that 
$$E'(b) = -5d+10d^2 + 4db = d(-5+10d+4b).$$
\noindent Since $b\le d$, we have that $E'(b)\le d(-5+14d)$ which is at most $0$ since $d\le 5/14$. Thus $E(b)$ is decreasing in $b$ on $[0,d]$ and so we have that $E(b) \ge E(0)$ as claimed.

Yet
$$E(0) = -1+5d-13d^2-12d^3 \le 1-5d \le 0,$$
\noindent since $d\in [0,1/5]$. Hence $E(b)$ and thus $F(b)$ are non-positive on $[0,d]$. Thus, $G(b) \le G(0)$ and hence $\hat{W}_{10}(b)\le \hat{W}_{10}(0)$ as desired.
\end{proof}

Thus we now have the following:
\begin{corollary}\label{CorLast}
$${\rm OPT}{\rm (P1)} \le {\rm OPT}{\rm (P10)} = \hat{W}_{10}(0) = \frac{3d(1-d)}{(1-2d)^2}.$$
\end{corollary}
\noindent We are now ready to complete the proof of Theorem~\ref{MainWThm}.
\begin{proof}[Proof of Theorem~\ref{RealMain2}]
It suffices to prove that $${\rm OPT}{\rm (P1)}\le 1.$$ This follows from Corollary~\ref{CorLast} and the fact that
$$\frac{3d(1-d)}{(1-2d)^2} \le 1,$$
\noindent since $d = \frac{7 - \sqrt{21}}{14}$ is a root of 
$$7d^2-7d+1=0,$$
\noindent in the interval $[0, 0.25)$.
\end{proof}

\section{Further Directions}\label{Further}

It is natural to wonder if the value of $d = \frac{7 - \sqrt{21}}{14}$ could be improved upon using our method.  As this approach solves an optimization problem leading to an equation whose root on $[0, 0.25)$ is exactly this value of $d$, at first glance this does not seem possible. However, we believe that the values that the variables achieve at the maximum point are not realizable by any graph.  That is, we believe there may be additional bounds on neighborhood densities imposed by structural conditions of graphs.  With such additional bounds, the optimal value could change.  Results in this direction would be very interesting.\footnote{
In particular, it seems that $b=0$ can be improved upon. Simple calculations suggest that $b$ should in fact on average be at least $d \left(\frac{1-3d}{1-2d}\right)$ for the maximum point. Such a bound would (if our rough calculations are correct) lead to $d\approx 0.187$ (and hence $\delta(G) \approx 0.813n$). However, this argument only seems to work on \emph{average} over all $f$ and $e$; thus some additional averaging bound would need to be added to {\rm (P1)}. But then the nice symmetrization argument of Subsection~\ref{ReduceTo10} could no longer be applied. Instead some exchange argument would be necessary. We were unable to prove this, though computer optimization programs suggest that the new optimum would indeed give some improvement on our value.}

Nevertheless, we do not believe that our weight function $w_G$ as defined would prove the Nash-Williams Conjecture asymptotically as our calculations suggest that $w_G$ is not non-negative for the nearly extremal examples (e.g. the clique blow-up of $C_4$, the independent blow-up of $K_4$, etc.). Instead a different weighting seems to be needed.

One might think that adding an initial weight $w$ to the triangles and then applying edge-gadgets to satisfy the remaining demand of each edge might lead to some improvement. We remark that curiously  using any \emph{uniform} initial weight $w$ results in the same $w_G(T)$ as the $w$ terms cancel out.  For ease of reading, we opted to not use any initial weight, equivalently setting $w=0$. However, it may be possible to improve the value of $d$ by using some non-uniform initial weight for triangles as Montgomery~\cite{Montgomery} did for general $r$. 

We think though that the key to solving the Nash-Williams Conjecture asymptotically may lie in choosing a \emph{non-uniform delegation}, that somehow edges should delegate demand only to triangles with certain properties and so on for the $K_4$s and $K_5$s. Yet we were unable to determine what delegation rule would fit the known extremal examples.

\section{Acknowledgments}
The authors would like to thank Daniela K\"{u}hn and Deryk Osthus for useful suggestions and for pointing out that Corollary 1.4 in~\cite{band} could be improved immediately using the main result of this paper. The authors would also like to thank Ben Barber as well as the anonymous referees for helpful comments.


\begin{thebibliography}{99}

\bibitem{tri17}
B.~Barber, D.~K\"{u}hn, A.~Lo, R.~Montgomery, and D.~Osthus,
\newblock{Fractional clique decompositions of dense graphs and hypergraphs}, 
\newblock{\em Journal of Combinatorial Theory Series B}, 127, 2017, 148--186.

\bibitem{tri16}
B.~Barber, D.~K\"{u}hn, A.~Lo, and D.~Osthus, 
\newblock{Edge-decompositions of graphs with high minimum degree}, 
\newblock{\em Advances in Mathematics}, 288, 2016, 337--385.

\bibitem{min}
B.~Barber, S.~Glock, D.~K\"{u}hn, A.~Lo, R.~Montgomery, and D.~Osthus,
\newblock{Minimalist designs},
\newblock{\em Random Structures and Algorithms}, 57, 2020, 47--63.

\bibitem{band}
P.~Condon, J.~Kim, D.~K\"{u}hn, and D.~Osthus,
\newblock{A bandwidth theorem for approximate decompositions},
\newblock{\em Proceedings London Mathematical Society}, 118, 2019, 1393--1449.

\bibitem{D15}
F.~Dross,
\newblock{Fractional triangle decompositions in graphs with large minimum degree},
\newblock{\em SIAM Journal on Discrete Mathematics}, 30(1), 2015, 36--42.

\bibitem{Dukes}
P.~Dukes,
\newblock{Rational decomposition of dense hypergraphs and some related eigenvalue estimates},
\newblock{\em Linear Algebra and Its Applications}, 436(9), 2012, 3736--3746.

\bibitem{DH20}
P.J.~Dukes and D.~Horsley, 
\newblock{On the minimum degree required for a triangle decomposition},
\newblock{SIAM Journal on Discrete Mathematics}, 2020, 597--610.

\bibitem{G14}
K.~Garaschuk. 
\newblock{Linear methods for rational triangle decompositions}, 
\newblock{PhD thesis}, University of Victoria, 2014.

\bibitem{decomp}
S.~Glock, D.~K\"{u}hn, A.~Lo, R.~Montgomery, and D.~Osthus,
\newblock{On the decomposition threshold of a given graph},
\newblock{\em Journal of Combinatorial Theory Series B}, 139, 2019, 47--127.

\bibitem{arb}
S.~Glock, D.~K\"{u}hn, A.~Lo, and D.~Osthus,
\newblock{The existence of designs via iterative absorption: Hypergraph F-designs for arbitrary F},
\newblock{To appear in {\em Memoirs of the American Mathematical Society}}, https://arxiv.org/pdf/1611.06827.pdf.



\bibitem{BCC}
S.~Glock, D.~K\"{u}hn, and D.~Osthus,
\newblock{Extremal Aspects of Graph and Hypergraph Decomposition Problems},
\newblock{Submitted}, https://arxiv.org/pdf/2008.00926.pdf.


\bibitem{HR}
P.E.~Haxell and V.~R\"{o}dl, 
\newblock{Integer and fractional packings in dense graphs}, 
\newblock{\em Combinatorica}, 21(1): 2001, 13--38.

\bibitem{Keev}
P.~Keevash, 
\newblock{The existence of designs}, 
\newblock{Submitted,} https://arxiv.org/pdf/1401.3665.pdf.

\bibitem{K}
T.P.~Kirkman, 
\newblock{On a problem in combinatorics}, 
\newblock{\em Cambridge Dublin Mathematical Journal}, 2, 1847, 191--204.

\bibitem{Montgomery}
R.~Montgomery,
\newblock{Fractional Clique Decompositions of Dense Graphs},
\newblock{\em Random Structures and Algorithms}, 54(4), 2019, 779--796.

\bibitem{NW}
C.S.J.~Nash-Williams,
\newblock{An unsolved problem concerning decomposition of graphs into
triangles}, 
\newblock{\em Combinatorial Theory and its Applications}, III, 1970, 1179--1183.

\bibitem{W}
R.M.~Wilson, 
\newblock{Decomposition of complete graphs into subgraphs isomorphic to a given graph}, 
\newblock{\em Congressus Numerantium XV}, 1975, 647--659.

\bibitem{Y05}
R.~Yuster, 
\newblock{Asymptotically optimal $K_k$-packings of dense
graphs via fractional $K_k$-decompositions},
\newblock{\em Journal of Combinatorial Theory Series B}, 95, 2005, 1--11.

\end{thebibliography}
\end{document}